\documentclass[11pt]{amsart}
\usepackage{amsmath}
\usepackage{amsfonts}
\usepackage{amssymb}
\newtheorem{teo}{Theorem}[section]
\newtheorem{lem}[teo]{Lemma}
\newtheorem{cor}[teo]{Corollary}
\newtheorem{prop}[teo]{Proposition}

\theoremstyle{remark}

\theoremstyle{definition}
\newtheorem{defi}[teo]{Definition}

\newcommand{\average}{{\mathchoice {\kern1ex\vcenter{\hrule height.4pt
width 6pt
depth0pt} \kern-9.7pt} {\kern1ex\vcenter{\hrule height.4pt width 4.3pt
depth0pt}
\kern-7pt} {} {} }}

\def\vint_#1{\mathchoice%
          {\mathop{\kern 0.2em\vrule width 0.6em height 0.69678ex depth -0.58065ex
                  \kern -0.8em \intop}\nolimits_{\kern -0.4em#1}}%
          {\mathop{\kern 0.1em\vrule width 0.5em height 0.69678ex depth -0.60387ex
                  \kern -0.6em \intop}\nolimits_{#1}}%
          {\mathop{\kern 0.1em\vrule width 0.5em height 0.69678ex
              depth -0.60387ex
                  \kern -0.6em \intop}\nolimits_{#1}}%
          {\mathop{\kern 0.1em\vrule width 0.5em height 0.69678ex depth -0.60387ex
                  \kern -0.6em \intop}\nolimits_{#1}}}
\def\vintslides_#1{\mathchoice%
          {\mathop{\kern 0.1em\vrule width 0.5em height 0.697ex depth -0.581ex
                  \kern -0.6em \intop}\nolimits_{\kern -0.4em#1}}%
          {\mathop{\kern 0.1em\vrule width 0.3em height 0.697ex depth -0.604ex
                  \kern -0.4em \intop}\nolimits_{#1}}%
          {\mathop{\kern 0.1em\vrule width 0.3em height 0.697ex depth -0.604ex
                  \kern -0.4em \intop}\nolimits_{#1}}%
          {\mathop{\kern 0.1em\vrule width 0.3em height 0.697ex depth -0.604ex
                  \kern -0.4em \intop}\nolimits_{#1}}}

\newcommand{\kintint}[2]{\mathchoice%
          {\mathop{\kern 0.2em\vrule width 0.6em height 0.69678ex depth -0.58065ex
                  \kern -0.8em \intop}\nolimits_{\kern -0.45em#1}^{#2}}%
          {\mathop{\kern 0.1em\vrule width 0.5em height 0.69678ex depth -0.60387ex
                  \kern -0.6em \intop}\nolimits_{#1}^{#2}}%
          {\mathop{\kern 0.1em\vrule width 0.5em height 0.69678ex depth -0.60387ex
                  \kern -0.6em \intop}\nolimits_{#1}^{#2}}%
          {\mathop{\kern 0.1em\vrule width 0.5em height 0.69678ex depth -0.60387ex
                  \kern -0.6em \intop}\nolimits_{#1}^{#2}}}

\makeatletter
\def\cleardoublepage{\clearpage\if@twoside \ifodd\c@page\else
\hbox{}
\thispagestyle{empty}
\newpage
\if@twocolumn\hbox{}\newpage\fi\fi\fi}
\makeatother
\title{Domain Variation Solutions for degenerate two phase free boundary problems}
\author{Aleksandr Dzhugan}
\address{Aleksandr Dzhugan: Dipartimento di Matematica\\ Universit\`a di Bologna\\ Piazza di Porta S.Donato 5\\ 40126, Bologna-Italy}
\email{aleksandr.dzhugan2@unibo.it }
\thanks{A.D. is supported by the fellowship INDAM-DP-COFUND-2015 "INdAM Doctoral Programme in
Mathematics and/or Applications Cofunded by Marie Sklodowska-Curie Actions", Grant 713485.}
\author{Fausto Ferrari}
\address{Fausto Ferrari: Dipartimento di Matematica\\ Universit\`a di Bologna\\ Piazza di Porta S.Donato 5\\ 40126, Bologna-Italy}
\email{fausto.ferrari@unibo.it }
\thanks{F.F. is partially supported by INDAM-GNAMPA project 2019: Propriet\`a di regolarit\`a delle soluzioni viscose con applicazioni a problemi di frontiera libera.}
\date{\today}

\begin{document}
\begin{abstract}
We discuss the domain variation solutions notion for some degenerate elliptic two-phase free boundary problems as well as the viscosity definition of the problem when the operator is degenerate.
\end{abstract}
\maketitle
\section{Introduction}
In the celebrated paper \cite{ACF}, the authors look for solutions, in the sense of the variation domain, to some nonlinear functionals. The minima to such functionals are endowed with some regularity properties like global Lipschitz continuity. In addition they also find the Euler-Lagrange equations that govern the underlying problem associated with the considered Bernoulli functional:  the so-called homogeneous elliptic two-phase free boundary problem. In our opinion the approach described in \cite{ACF} is highly not trivial and, at the same time, reveals some interesting details that we consider particularly useful for extending to more complex operators the geometrical approach that has been developed in studying regularity of free boundary problems since \cite{ACF} and \cite{C1}, see also \cite{CS} for a basic bibliography. In fact, after that seminal papers, there have been many other achievements about the regularity theory of the viscosity solutions of two phase free boundary problems. From this point of view, we recall for the regularity of the free boundary as well \cite{Feldman}, \cite{CFS}, \cite{FS_advan}, \cite{AF}, \cite{LN}  respectively for homogeneous fully nonlinear operators, homogeneous linear operators with variable coefficients, homogeneous linear operators, with bounded first order terms, homogeneous fully nonlinear operators with flat boundaries and homogeneous $p-$Laplace operator. Successively, after the fundamental contribution introduced in \cite{D},  many other inhomogeneous cases have been faced: see \cite{DFS_apdes}, \cite{DFS_japa}, \cite{DFS_trans}.  In fact the approach used in \cite{D} is particularly flexible and has been extended to other one phase inhomogeneous cases that are not covered by previous cited papers yet, see the recent progresses contained in \cite{Braga_Moreira}, \cite{Leitao_Ricarte} and \cite{Lederman_Wol_1}, \cite{Lederman_Wol_2} in the variational one phase case. Very recently also a new contribution following the main stream of \cite{ACF}, in the variational setting,   appeared:  see \cite{DSV}. We remark in addition that in \cite{K} a different approach, that does not use the monotonicity formula in the $p-$Laplace setting, has been introduced. Concerning the notion of viscosity solution we refer in any case to \cite{CC}, \cite{CIL} and \cite{BCD}.

 In this note, following the ideas described in \cite{ACF}, we would like to formalize the correct formulation of the two phase free boundary problems arising from Bernoulli type functionals when we consider nonnegative matrix of variable coefficients as well as re-find the cases in which a nonlinear dependence on the gradient of the solution itself exists. We have in mind two concrete examples respectively given by the Kohn-Laplace operator in the Heisenberg group and the $p(x)-$Laplace operator. The $p-$Laplace case has been also discussed in \cite{LDT} so that it results also interesting to understand the $p(x)-$Laplace operator as well. We remind that the $p-$Laplace operator is
$$
\Delta_p:=\mbox{div}(|\nabla \cdot |^{p-2}\nabla),
$$ 
while the $p(x)-$Laplace operator is
$$
\Delta_{p(x)}:=\mbox{div}(|\nabla \cdot |^{p(x)-2}\nabla),
$$
where the function $p$ satisfies $1<p(x)<\infty.$ Of course, $\Delta_{p(x)}=\Delta_{p}$ when $p(x)$ is constant and $p(x)\equiv p.$

The Kohn-Laplace operator in $\mathbb{H}^1$ is defined as
\begin{equation}\label{equa}
\Delta_{\mathbb{H}^1}u(x,y,t)=\frac{\partial^2u}{\partial x^2}+\frac{\partial^2u}{\partial y^2}+4y\frac{\partial^2u}{\partial x\partial t}-4x\frac{\partial^2u}{\partial y\partial t}+4(x^2+y^2)\frac{\partial^2u}{\partial t^2}
\end{equation}
and even if it is linear, it results to be degenerate elliptic. In particular, using an intrinsic interpretation of the geometric object entering in the description of the non-commutative underlying structure $\mathbb{H}^1$ it is possible to obtain an intrinsic formulation of the two phase problem.
The Kohn-Laplace operator is degenerate. In fact its lowest eigenvalue is always zero. As a consequence it is important to understand what is the right condition to require to put on the free boundary in case we wish to formulate the problem in a viscosity sense. 
The theory of the viscosity solutions has been applied to the study of free boundary problems like:
\begin{equation}\label{two_phase_Euc}
\left\{\begin{array}{lr}
\Delta u=f,& \mbox{in }\Omega^+(u):= \{x\in \Omega:\hspace{0.1cm} u(x)>0\},\\
\Delta u=f,& \mbox{in }\Omega^-(u):=\mbox{Int}(\{x\in \Omega:\hspace{0.1cm} u(x)\leq 0\}),\\
(u^+_n)^2-(u^-_n)^2=1&\mbox{on }\mathcal{F}(u):=\partial \Omega^+(u)\cap \Omega,
\end{array}
\right.
\end{equation} 
since \cite{C1}, for homogeneous problems, by Luis Caffarelli. 
Here $\Omega\subset \mathbb{R}^n$ is an open set, and $f\in C^{0,\alpha}\cap L^\infty(\Omega),$ while $u^+_n$ formally denotes the normal derivative at the points belonging to  $\mathcal{F}(u),$ where $n$ is the unit normal in those points
 whenever this makes sense, pointing inside $\Omega^+(u),$ as well as $u^-_n$ denotes the normal derivative to the set  $\mathcal{F}(u)$ and $n$ is the unit normal to the set $\mathcal{F}(u)$ at the point $x\in \mathcal{F}(u)$ pointing inside $\Omega^-(u)$. 
 If in case $\mathcal{F}(u)$ were $C^1,$ even supposing for simplicity that $f\equiv 0,$ then $u$ would satisfy $\Delta u=0$ in $\Omega^+(u)\cup\Omega^-(u). $ On the other hand, $u\in C(\Omega)$ is a viscosity solution, so that $\Delta u=0$ in $\Omega^+(u)$ and $\Delta u=0$ in $\Omega^-(u)$ in the classic sense and the problem \eqref{two_phase_Euc}  may be reduced to two Dirichlet problems. However the assumption on the level set $\mathcal{F}(u):=\partial \Omega^+(u)\cap \Omega$ can not be formulated in a classical fashion, because $\mathcal{F}(u)$ is an unknown of the problem. In principle, the set $\mathcal{F}(u)$ might be very irregular and the notion of solution would not make sense in the classical meaning, so that has to be weakened. 
 On the contrary, we suppose exactly that the fact itself of knowing that $u$ satisfies the free boundary problem should imply that $u$ is endowed with some further regularity properties.  Thus, assuming only that $\mathcal{F}(u)$ is Lipschitz, the solution of the Dirichlet problem in a neighborhood of the free boundary may be {\it a priori} no better than a H\"older continuous function until the boundary. Hence, it appears clear that we can not give a pointwise classical formulation of the problem on the free boundary. For avoiding this loop, in \cite{C1} a viscosity notion of solution was introduced. In that case the boundary condition is supposed to be fulfilled only where a weak normal exists, see \cite{CS}. 

The definition of solution in a viscosity sense of the problem \eqref{two_phase_Euc} can be stated, see \cite{DFS_apdes} and also the original statement in \cite{C1} or in \cite{CS}, in the following way: a continuous function $u$ is a solution to \eqref{two_phase_Euc} if:
\begin{itemize}
\item[(i)] $\Delta u=f$ in a viscosity sense in $\Omega^+(u)$ and $\Omega^-(u);$
\item[(ii)] let $x_0\in \mathcal{F}(u).$ For every function $v\in C(B_\varepsilon(x_0))$, $\varepsilon >0$ such that $v\in C^{2}(\overline{B^+(v)})\cap C^{2}(\overline{B^-(v)}),$ being $B:= B_\varepsilon(x_0)$  and $\mathcal{F}(v)\in C^2,$ if $v$ touches $u$ from below (resp. above) at $x_0\in \mathcal{F}(v),$ then 
$$
(v_n^+(x_0))^2-(v_n^-(x_0))^2\leq 1\quad (\mbox{resp.}\quad (v_n^+(x_0))^2-(v_n^-(x_0))^2\geq 1).
$$
\end{itemize}

Moreover, also the following notion of strict comparison subsolution (supersolution) plays a fundamental role in the regularity theory of one/two-phase free boundary problems, see \cite{D}:
a function $v\in C(\Omega)$ is a strict comparison subsolution (supersolution) to \eqref{two_phase_Euc} if: $v\in C^{2}(\overline{\Omega^+(v)})\cap C^{2}(\overline{\Omega^-(v)})$ and
\begin{itemize}
\item[(i)] $\Delta v>f$ (resp. $\Delta v<f$) in a viscosity sense in $\Omega^+(v)\cup\Omega^-(v);$
\item[(ii)] for every $x_0\in \Omega,$ if $x_0\in \mathcal{F}(v)$ then 
$$
(v_n^+(x_0))^2-(v_n^-(x_0))^2> 1\quad (\mbox{resp.}\quad (v_n^+(x_0))^2-(v_n^-(x_0))^2<1,\quad v_n^+(x_0)\not=0).
$$
\end{itemize}

As a consequence a strict comparison subsolution $v$ cannot touch a viscosity solution $u$ from below at any point in $\mathcal{F}(u)\cap\mathcal{F}(v).$ Analogously a strict comparison supersolution $v$ cannot touch a viscosity solution $u$ from above at any point in $\mathcal{F}(u)\cap\mathcal{F}(v).$ When $u$ is a classical solution and the free boundary is sufficiently smooth,  previous comparison property comes from the Hopf maximum principle, whenever the condition on the flux balance on the free boundary is given by a function $G(\cdot,\cdot)$ defined in $[0,\infty)\times[0,\infty)$  that is also strictly increasing in the first entrance and strictly decreasing in the second variable.

 We are mainly interested in viscosity solution, but the natural definition of two-phase free boundary problems is usually determined by looking for local minima of functionals like 
\begin{equation}\label{omo}
\mathcal{E}(v)=\int_{\Omega}\Big(|\nabla v|^2+\chi_{\{v>0\}}+2fv\Big)dx
\end{equation}
defined on subsets of $H^1(\Omega)$ satisfying some fixed conditions, for instance assumed on the boundary of $\Omega$ and on the sign of the functions themselves. In \cite{ACF} exactly this approach has been followed for functionals, associated with the Laplace operator like \eqref{omo}, in the homogeneous case. As a consequence, for local minima $u$ of \eqref{omo} (supposing $f=0$) in \cite{ACF} have been determined the conditions that have to be satisfied on the free boundary, morally the set $\{x\in \Omega:\: u(x)=0\}$.
 
Since we are interested in problems governed by other operators with respect to the laplacian, like nonlinear ones and, overall possibly degenerate, we wish, at first, to understand what is the right condition to put on the free boundary, for the problem in a non-divergence form, in a degenerate setting.

In fact the free boundary $\mathcal{F}(u)$ is an unknown of the problem and for this reason we need to start from the energy functional that describes the problem in the variational setting for obtaining the non-divergence case.

 With this aim, we discuss the notion of domain variation solution assuming that the energy functionals that we wish to study may be associated with degenerate operators like the $p(x)$-Laplace operator $\Delta_{p(x)},$ that is a generalization of the most popular $p-$Laplace operator when the function $p(x)$ is constant or operators like $\mbox{div}(A(x)\nabla),$ supposing that the matrix $A$ satisfies $\langle A(x)\xi,\xi\rangle\geq 0$ for every $\xi\in \mathbb{R}^n$ whenever $A$ is a smooth matrix of coefficients. For the notion of solution in the sense of variation domain and applications we refer to \cite{Weiss}, see also \cite{Fe}.  
 At the end of our discussion we conclude that in any Carnot group the two phase problem assumes the following nonvariational form:
 \begin{equation}\label{two_phase_Carnot_x}
\left\{\begin{array}{lr}
\Delta_{\mathbb{G}} u=f,& \mbox{in }\Omega^+(u):= \{x\in \Omega:\hspace{0.1cm} u(x)>0\},\\
\Delta_{\mathbb{G}} u=f,& \mbox{in }\Omega^-(u):=\mbox{Int}(\{x\in \Omega:\hspace{0.1cm} u(x)\leq 0\}),\\
|\nabla_{\mathbb{G}} u^+|^2-|\nabla_{\mathbb{G}} u^-|^2=1,&\mbox{on }\mathcal{F}(u):=\partial \Omega^+(u)\cap \Omega,
\end{array}
\right.
\end{equation}
where $\Delta_{\mathbb{G}}$ is a sublaplacian in a Carnot group $\mathbb{G},$ see Section \ref{keytool} for the definitions of Carnot groups and the associated notation, and Section \ref{variational_x} for a little more general presentation of the result. We remark here, however, that now the condition posed on free boundary is governed by an {\it intrinsic} jump of gradients, see Section \ref{keytool} and, for the one-phase case, see \cite{FeVa}.


Moreover, in the parallel case of the $p(x)-$Laplacian, the functional becomes
$$
\mathcal{E}_{p(x)}(u)=\int_{\Omega}\left(\mid \nabla v\mid^{p(x)} +\chi_{\{v>0\}}+p(x)fv\right)dx,
$$
and in this case we obtain:
 \begin{equation}\label{two_phase_p(x)}
\left\{\begin{array}{lr}
\Delta_{p(x)} u=f,& \mbox{in }\Omega^+(u):= \{x\in \Omega:\hspace{0.1cm} u(x)>0\},\\
\Delta_{p(x)} u=f,& \mbox{in }\Omega^-(u):=\mbox{Int}(\{x\in \Omega:\hspace{0.1cm} u(x)\leq 0\}),\\
|\nabla u^+|^{p(x)}-|\nabla u^-|^{p(x)}=\frac{1}{p(x)-1},&\mbox{on }\mathcal{F}(u):=\partial \Omega^+(u)\cap \Omega.
\end{array}
\right.
\end{equation}
see Section \ref{pxlapvar} for a slightly more general setting of the problem.

We complete our analysis in Section \ref{conclusions} stating the good notion of viscosity solutions for problems like \eqref{two_phase_Carnot_x} and \eqref{two_phase_p(x)}. In the case \eqref{two_phase_Carnot_x} the characteristic points introduce new difficulties in the application of the approach applied in \cite{D}.

In the next section, for describing the notion of domain variation solution, we start with the simplest case in one dimension.


\section{The simplest one dimension Euclidean case}
Before entering into the details of our subject, we consider the basic heuristic example in the one dimension for the following functional
$$
\mathcal{E}(v)=\int_{-1}^{1}(v'^2+\chi_{\{v>0\}}+2fv)dx,
$$ 
where
$$
\chi_{\{v>0\}}=\left\{
\begin{array}{l}
1,\quad x\in \{v>0\},\\
0,\quad x\in \{v\leq 0\},
\end{array}
\right.
$$
and $v\in K=\{w\in H^1([-1,1]):\quad w(-1)=a,\quad w(1)=b\}$ being $a, b$ assigned values to the boundary. Moreover we assume, for simplicity, that $f\in C^{0,\gamma}([-1,1]).$

We are interested in those functions which become minima or critical values for $\mathcal{E}$ perturbing the set of definition in a neighborhood of the points where $v$ vanishes. In mathematical language, for every function $v\in K$ and for every function $\varphi\in C_0^\infty(]-1,1[)$ we consider the function $v(x)=v_\varepsilon^\varphi(x+\varepsilon \varphi(x)).$ We shall simply write $v_\varepsilon:=v_\varepsilon^\varphi$ to avoid the cumbersome notation. It is clear that $\tau_\varepsilon=I+\varepsilon \varphi$ is an application that transforms $[-1,1]$ in  itself whenever $\varepsilon$ is sufficiently small. We say that $v$ is a variational domain solution whenever 
$$
\frac{d}{d\varepsilon}\mathcal{E}(v_\varepsilon)_{|\varepsilon=0}=0.
$$
To do this, we consider
\begin{equation}\label{unadimensione}
\mathcal{E}(v_\varepsilon)=\int_{-1}^{1}v'^2_\varepsilon(y)+\chi_{\{v_\varepsilon>0\}}(y)+2fv_\varepsilon dy
\end{equation}
Since $\tau_\varepsilon$ is invertible whenever $\varepsilon$ is small we obtain
 $(\tau^{-1}_\varepsilon)'(y)=(\tau'_\varepsilon(x))^{-1},$ being $x=\tau^{-1}_\varepsilon(y)$ and 
$$
\tau'_\varepsilon(x)=1+\varepsilon\varphi'(x),
$$
$$
(\tau^{-1}_\varepsilon)'(y)=\frac{1}{1+\varepsilon\varphi'(\tau^{-1}_\varepsilon(y))}.
$$
This implies that for $\varepsilon\to 0$
$$
(\tau^{-1}_\varepsilon)'(y)=1-\varepsilon\varphi'(\tau^{-1}_\varepsilon(y))+o(\varepsilon).
$$
We perform the change of variable $y=\tau_\varepsilon(x)$  so that:
\begin{equation}\label{unadimensione1}
\begin{split}
&\mathcal{E}(v_\varepsilon)=\int_{-1}^{1}\left(v'^2_\varepsilon(\tau_\varepsilon(x))+\chi_{\{v_\varepsilon>0\}}(\tau_\varepsilon(x))+2f(\tau_\varepsilon(x))v_\varepsilon(\tau_\varepsilon(x))\right)\tau_\varepsilon'(x)dx\\
&=\int_{-1}^{1}\left(v'^2_\varepsilon(\tau_\varepsilon(x))+\chi_{\{v_\varepsilon>0\}}(\tau_\varepsilon(x))+2f(\tau_\varepsilon(x))v_\varepsilon(\tau_\varepsilon(x))\right)(1+\varepsilon \varphi'(x))dx
\end{split}
\end{equation}
and, since $v'(x)=v_\varepsilon'(\tau_{\varepsilon}(x))\tau_\varepsilon'(x)=v_\varepsilon'(\tau_{\varepsilon}(x))(1+\varepsilon\varphi'(x))$, we get
\begin{equation}\label{unadimensione3}
\begin{split}
&=\int_{-1}^{1}[v'^2(x)(1+\varepsilon\varphi'(x))^{-2}+\chi_{\{v_\varepsilon>0\}}(\tau_\varepsilon(x))+2f(\tau_\varepsilon(x))v(x)](1+\varepsilon\varphi'(x))dx\\
\end{split}
\end{equation}
\begin{equation*}
\begin{split}
&=\int_{-1}^{1}[v'^2(x)(1-\varepsilon\varphi'(x)+o(\varepsilon))^2+\chi_{|\{v_\varepsilon>0\}}(x+\varepsilon \varphi(x))+2f(\tau_\varepsilon(x))v(x)](1+\varepsilon\varphi'(x))dx\\
&=\int_{-1}^{1}[v'^2(x)(1-2\varepsilon\varphi'(x)+o(\varepsilon))+\chi_{|\{v_\varepsilon>0\}}(x+\varepsilon \varphi(x))](1+\varepsilon\varphi'(x))dx\\
&+2\int_{-1}^1f(\tau_\varepsilon(x))v(x)(1+\varepsilon\varphi'(x))dx
\end{split}
\end{equation*}
that is
\begin{equation}\label{unadimensione4}
\begin{split}
&=\mathcal{E}(v)+\int_{-1}^{1}-\varepsilon v'^2(x)\varphi'(x)+[\chi_{|\{v_\varepsilon>0\}}(x+\varepsilon \varphi(x))(1+\varepsilon\varphi'(x))-\chi_{|\{v>0\}}(x)]dx\\
&+2\varepsilon\int_{-1}^1(f(x)v(x)\varphi'(x)+ f'(x)v(x)\varphi(x))dx+o(\varepsilon)\\
&=\mathcal{E}(v)+\int_{-1}^{1}-\varepsilon v'^2(x)\varphi'(x)+[\chi_{|\{v_\varepsilon>0\}}(x+\varepsilon \varphi(x))-\chi_{|\{v>0\}}(x)]dx\\
&+\varepsilon\int_{-1}^{1}\chi_{|\{v_\varepsilon>0\}}(x+\varepsilon \varphi(x))\varphi'(x)dx+2\varepsilon\int_{-1}^1(f(x)\varphi'(x)+ f'(x)\varphi(x))v(x)dx+o(\varepsilon)\\
\end{split}
\end{equation}
and, integrating by parts and recalling that $\varphi$ is a compactly supported function, we obtain:
\begin{equation}\label{unadimensione5}
\begin{split}
&=\mathcal{E}(v)+\int_{-1}^{1}-\varepsilon v'^2(x)\varphi'(x)dx+\varepsilon\int_{-1}^{1}\chi_{|\{v_\varepsilon>0\}}(x+\varepsilon \varphi(x))\varphi'(x)dx\\
&+2(f(1)\varphi(1)v(1)-f(-1)\varphi(-1)v(-1))-2\varepsilon\int_{-1}^1f(x)\varphi(x)v'(x)dx+o(\varepsilon)\\
&=\mathcal{E}(v)+\int_{-1}^{1}-\varepsilon v'^2(x)\varphi'(x)dx+\varepsilon\int_{-1}^{1}\chi_{|\{v_\varepsilon>0\}}(x+\varepsilon \varphi(x))\varphi'(x)dx\\
&-2\varepsilon\int_{-1}^1f(x)\varphi(x)v'(x)dx+o(\varepsilon).
\end{split}
\end{equation}
As a consequence, if $v$ is a local minimum for the functional $\mathcal{E}$ in $K,$ then

\begin{equation}\label{unadimensione2}
\begin{split}
&0\leq \frac{\mathcal{E}(v_\varepsilon)-\mathcal{E}(v)}{\varepsilon}=-\int_{-1}^{1}v'^2(x)\varphi'(x)dx+\int_{-1}^{1}\chi_{|\{v_\varepsilon>0\}}(x+\varepsilon \varphi(x))\varphi'(x)dx\\
&-2\int_{-1}^1f(x)\varphi(x)v'(x)dx+o(1).
\end{split}
\end{equation}
Moreover, it also results that for every $\varphi\in C_0^{\infty}(]-1,1[)$ we have
$$
\lim_{\varepsilon\to 0}\frac{\mathcal{E}(v_\varepsilon)-\mathcal{E}(v)}{\varepsilon}=0.
$$
Hence, if $v$ is a local minimum for $\mathcal{E}$ on $K$, then $v$ is a domain variational solution.  

As a consequence, we have obtained that a local minimum have to satisfy the following relationship:
$$
-\int_{-1}^{1}v'^2(x)\varphi'(x)dx-2\int_{-1}^1f(x)\varphi(x)v'(x)dx+\int_{-1}^{1}\chi_{|\{v>0\}}(x)\varphi'(x)dx=0,
$$
for every $\varphi\in C_0^{\infty}(]-1,1[).$

On the other hand, for every $\phi \in C^{\infty}_0(-1,1)$ such that $\mbox{supp}(\phi)\subset \{v_\varepsilon>0\}$ or $\mbox{supp}(\phi)\subset \mbox{int}\{v_\varepsilon\leq 0\}$
it follows from the previous relation that $v''=f(x)$ in $]-1,1[\setminus \{x\in ]-1,1[:\quad v(x)=0\}$ because $v$ is a local minimum for $\mathcal{E},$ (we will proof this property below in a more general case).

As a consequence
\begin{equation}\label{equauno}\begin{split}
&-\lim_{\delta\to 0^+}\int_{-1}^{x_\delta}v'^2(x)\varphi'(x)+2f(x)\varphi(x)v'(x)dx-\lim_{\epsilon\to 0^+}\int_{x_\epsilon}^{1}v'^2(x)\varphi'(x)+2f(x)\varphi(x)v'(x)dx\\
&-\lim_{\delta\to 0^+,\delta\to 0^+}\int_{x_\epsilon}^{x_\delta}v'^2(x)\varphi'(x)+2f(x)\varphi(x)v'(x)dx+\int_{-1}^{1}\chi_{|\{v>0\}}(x)\varphi'(x)dx=0,
\end{split}
\end{equation}
for every $\varphi\in C_0^{\infty}(]-1,1[)$ and $\epsilon,\delta>0,$ we consider the sets $\{v(x)<-\epsilon\}$ and $\{v(x)>\delta\}.$ 
Then integrating by parts we obtain, from \eqref{equauno} and keeping in mind that we assumed $\mbox{meas}_1(\{v=0\})=0$:
\begin{equation}\label{equadue}\begin{split}
&\lim_{\delta\to 0^+}\int_{-1}^{x_\delta}2\left(v''(x)-f(x))\right)\varphi(x)v'(x)dx-\lim_{\delta\to 0^+}[v'^2(x)\varphi(x)]_{x=-1}^{x=x_\delta}\\
&+\lim_{\epsilon\to 0^+}\int_{\delta_\epsilon}^{1}2\left(v''(x)-f(x))\right)\varphi(x)v'(x)dx-\lim_{\epsilon\to 0^+}[v'^2(x)\varphi(x)]^{x=1}_{x=x_\epsilon}\\
&+\int_{-1}^{1}\chi_{|\{v>0\}}(x)\varphi'(x)dx=0,
\end{split}
\end{equation}
where $v(x_\epsilon)=-\epsilon$ and $v(x_\delta)=\delta.$

Thus, from \eqref{equadue}, we get 
\begin{equation}\begin{split}
-\lim_{\delta\to 0^+}[v'^2(x)\varphi(x)]_{x=-1}^{x=x_\delta}-\lim_{\epsilon\to 0^+}[v'^2(x)\varphi(x)]^{x=1}_{x=x_\epsilon}+\int_{-1}^{1}\chi_{|\{v>0\}}(x)\varphi'(x)dx=0,
\end{split}
\end{equation}
or
\begin{equation}\begin{split}
-\lim_{\delta\to 0^+}v'^2(x_\delta)\varphi(x_\delta)+\lim_{\epsilon\to 0^+}v'^2(x_\epsilon)\varphi(x_\epsilon)+\int_{x_0}^{1}\varphi'(x)dx=0,
\end{split}
\end{equation}
that implies, for every $\varphi \in C_0^{\infty},$ denoting by $x_0$ the free boundary, that is $v(x_0)=0:$

$$
(-(v^-)'^2(x_0)+(v^+)'^2(x_0))\varphi(x_0)-\varphi(x_0)=0.
$$
Hence, it results
$$
(v^+)'^2(x)-(v^-)'^2(x)=1, \quad \mbox{on}\quad \{v=0\}.
$$
In this way, we have obtained the free boundary condition associated with the "Euler-Lagrange" equations to local minima of the functional $\mathcal{E}$ in the non-homogeneous  case, (of course assuming that the free boundary is a set of measure zero).
We also proved that, at least in one dimension, the free boundary condition does not depend on the non-homogeneous term $f$.


\section{Basic informations about the Heisenberg group and Carnot groups}\label{keytool}
Let $\mathbb{H}^n$ be the Heisenberg group of order $n.$ We denote by
$\mathbb{H}^n$ the set $\mathbb{R}^{2n+1},$ $n\in \mathbb{N},$ $n\geq 1,$ $(x,y,t)\in \mathbb{R}^{2n+1},$ endowed with the non-commutative inner law such that for every $(x_1,y_1,t_1)\in \mathbb{R}^{2n+1},$ $(x_2,y_2,t_2)\in \mathbb{R}^{2n+1},$ $x_{i}\in \mathbb{R}^n,$ $y_{i}\in \mathbb{R}^n,$ $i=1,2:$ 
$$
(x_1,y_1,t_1)\circ (x_2,y_2,t_2)=(x_1+x_2,y_1+y_2,t_1+t_2+2(x_2\cdot y_1-x_1\cdot y_2)),
$$
and $x_i\cdot y_i$ denote the usual inner product in $\mathbb{R}^n.$

Let $X_i=(e_i,0,2y_i)$ and $Y_i=(0,e_i,-2x_i),$ $i=1,\dots,n,$ where $\{e_i\}_{1\leq i\leq n}$ is the canonical basis for $\mathbb{R}^n.$

We use the same symbol to denote the vector fields associated with the previous vectors so that for $i=1,\dots,n$
$$
X_i=\partial_{x_i}+2y_i\partial_t
$$
$$
Y_i=\partial_{y_i}-2x_i\partial_t.
$$
The commutator between the vector fields is
$$
[X_i,Y_i]=-4\partial_t
$$
otherwise is $0.$ 
The intrinsic gradient of a smooth function $u$ in a point $P$ is  
$$
\nabla_{\mathbb{H}^n}u(P)=\sum_{i=1}^n(X_iu(P)X_i(P)+Y_iu(P)Y_i(P)).
$$
There exists a unique metric on $H\mathbb{H}^n(P)=\mbox{span}\{X_1,\dots,X_n,Y_i,\dots,Y_n\}$ that makes orthonormal the set of vectors $\{X_1,\dots,X_n,Y_i,\dots,Y_n\}.$ Thus, for every $P\in \mathbb{H}^n$ and  for every $U,W\in H\mathbb{H}^n(P),$ $U=\sum_{j=1}^n(\alpha_{1,j}X_{j}(P)+\beta_{1,j}Y_j(P)),$
$V=\sum_{j=1}^n(\alpha_{2,j}X_{j}(P)+\beta_{2,j}Y_j(P))$
$$
\langle U,V\rangle=\sum_{j=1}^n(\alpha_{1,j}\alpha_{2,j}+\beta_{1,j}\beta_{2,j}).
$$  
In particular we get a norm associated with the metric on $\mbox{span}\{X_1,\dots,X_n,Y_i,\dots,Y_n\}$  and
$$
|U|=\sqrt{\sum_{j=1}^n\left(\alpha_{1,j}^2+\beta_{1,j}^2\right)}.
$$
For example, the norm of the intrinsic gradient of the smooth function $u$ in $P$ is
$$
|\nabla_{\mathbb{H}^n} u(P)|=\sqrt{\sum_{i=1}^n\left((X_iu(P))^2+(Y_iu(P))^2\right)}.
$$
Moreover, if $\nabla_{\mathbb{H}^n} u(P)\not=0$ the norm of
$$
\frac{\nabla_{\mathbb{H}^n}u(P)}{|\nabla_{\mathbb{H}^n}u(P)|}
$$
is equal to one.

If $\nabla_{\mathbb{H}^n} u(P)=0$ then we say that the point $P$ is characteristic for the smooth surface $\{u=u(P)\}.$
Hence for every point $P\in \{u=u(P)\},$  that it is not characteristic, it is well defined the intrinsic normal to the surface $\{u=u(P)\}$ as follows: 
$$
\nu(P)=\frac{\nabla_{\mathbb{H}^n} u(P)}{|\nabla_{\mathbb{H}^n} u(P)|}.
$$
We introduce in the Heisenbeg group $\mathbb{H}^n$ the following gauge norm:
$$
d_G(x,y,t)\equiv||(x,y,t)||=\sqrt[4]{(|x|^2+|y|^2)^2+t^2}.
$$
In particular for every positive number $r$ the gauge ball of radius $r$ centerd in $0$ is
$$
B(0,r)=\{P\in \mathbb{H}^n :\:\:\|P\|<r\}.
$$
In the Heisenberg group a group of dilation is also defined as follows: for every $r>0$ and for every $P\in \mathbb{H}^n$ let
$$
\delta_r(P)=(rx,ry,r^2t).
$$

Let $(\xi,\eta,\sigma)\in\mathbb{H}^n,$ then
 $$d_G(\xi,\eta,t)=\sqrt[4]{(|\xi|^2+|\eta|^2)^2+\sigma^2}=d_K((\xi,\eta,\sigma),(0,0,0)).$$ 
 In particular, for every $i=1,\dots,n$,
\begin{equation}\begin{split}
X_id_G&=\frac{1}{4}((| \xi|^2+|\eta|^2)^2+\sigma^2)^{-\frac{3}{4}}(4(|\xi|^2+|\eta|^2)\xi_i+4\sigma\eta_i)=\frac{1}{4}d_G^{-3}(4(|\xi|^2+| \eta|^2)\xi_i+4\sigma\eta_i)\\
&=d_G^{-3}((|\xi|^2+|\eta|^2)\xi_i+\sigma\eta_i)
\end{split}
\end{equation}
and
$$
Y_id_G=\frac{1}{4}((|\xi|^2+|\eta|^2)^2+t^2)^{-\frac{3}{4}}(4(|\xi|^2+|\eta|^2)\eta_i-4\sigma\xi_i)= d_G^{-3}( (|\xi|^2+|\eta|^2)\eta_i- \sigma\xi_i).
$$
Moreover, for every $i=1,\dots,n:$
$$
X_i^2d_G=-3d_G^{-7}((|\xi|^2+|\eta|^2)\xi_i+\sigma\eta_i)^2+d_G^{-3}(2\xi_i^2+(|\xi|^2+|\eta|^2)+2\eta_i^2)
$$
and
$$
Y_i^2d_G=-3d_G^{-7}((|\xi|^2+|\eta|^2)\eta_i-\sigma\xi_i)^2+d_G^{-3}(2\eta_i^2+(|\xi|^2+|\eta|^2)+2\xi_i^2).
$$
As a consequence,
\begin{equation}\begin{split}
&|\nabla_{\mathbb{H}^n}d_G|^2=\sum_{i=1}^n\left((X_id_G)^2+(Y_id_G)^2\right)=d_G^{-6}\sum_{i=1}^n\left( (|\xi|^2+|\eta|^2)^2(\xi_i^2+\eta_i^2)+\sigma^2(\xi_i^2+\eta_i^2)\right)\\
&=d_G^{-6}\left((|\xi|^2+|\eta|^2)^3+\sigma^2(|\xi|^2+|\eta|^2)\right)
=(|\xi|^2+|\eta|^2)d_G^{-2},
\end{split}
\end{equation}
and
\begin{equation}\begin{split}
\Delta_{\mathbb{H}}d_G&=-3d_G^{-7}\left((|\xi|^2+|\eta|^2)^3+\sigma^2(|\xi|^2+|\eta|^2)\right)+(2n+4)d_G^{-3}(|\xi|^2+|\eta|^2)\\
&=(2n+1)(|\xi|^2+|\eta|^2)d_G^{-3}.
\end{split}
\end{equation}

Following the above calculations, for every $i=1,\dots,n$, we have
$$
X_id_G^{2-Q}=(2-Q)d_G^{1-Q}\left(d_G^{-3}((|\xi|^2+|\eta|^2)\xi_i+\sigma\eta_i)\right)
$$
and

$$
Y_id_G^{2-Q}=(2-Q)d_G^{1-Q}\left( d_G^{-3}( (|\xi|^2+|\eta|^2)\eta_i- \sigma\xi_i)\right).
$$
Thus
\begin{equation}\begin{split}
\Delta_{\mathbb{H}}d_G^{2-Q}&=(2-Q)(1-Q)d_G^{-Q}
\sum_{i=1}^n\left((X_id_G)^2+(Y_id_G)^2\right)+(2-Q)d_G^{1-Q}\sum_{i=1}^n\left(X_i^2d_G+Y_i^2d_G\right)\\
&=(2-Q)(1-Q)d_G^{-Q}
\mid\nabla_{\mathbb{H}^n}d_G\mid^2+(2-Q)d_G^{1-Q}\Delta_{\mathbb{H}^n}d_G\\
&=(2-Q)\left((1-Q)d_G^{-2-Q}(|\xi|^2+|\eta|^2) +d_G^{-2-Q}(2n+1)(|\xi|^2+|\eta|^2)\right)\\
&=(2-Q)d_G^{-2-Q}(|\xi|^2+|\eta|^2)\left(1-Q + 2n+1\right)=0,\end{split}
\end{equation}
whenever $Q=2n+2.$ That is, $d_G^{2-Q}$ is, up to a constant, the fundamental solution of the sublaplacian in the Heisenberg group.

We define the symmetrized horizontal Hessian matrix of the smooth function $u$ at $P$ the following $2n\times2n$ matrix:
$$
D_{\mathbb{H}^n}^{2*}u(P)=\left[\begin{array}{lrlr}
X_1^2u(P),&\dots,X_1X_nu(P),&\frac{X_1Y_1+Y_1X_1}{2}u(P),&\dots\frac{X_1Y_n+Y_nX_1}{2}u(P)\\
X_2X_1u(P),&\dots,X_2X_nu(P),&\frac{X_2Y_1+Y_1X_2}{2}u(P),&\dots\frac{X_2Y_n+Y_nX_2}{2}u(P)\\
\vdots\\
X_nX_1u(P),&\dots,X_n^2u(P),&\frac{X_nY_1+Y_1X_n}{2}u(P),&\dots\frac{X_nY_n+Y_nX_n}{2}u(P)\\
\frac{Y_1X_1+X_1Y_1}{2}u(P),&\dots,\frac{Y_1X_n+X_nY_1}{2}u(P),&Y_1^2u(P),&\dots,Y_1Y_nu(P)\\
\frac{Y_2X_1+X_1Y_2}{2}u(P),&\dots,\frac{Y_2X_n+X_nY_2}{2}u(P),&Y_2Y_1u(P),&\dots,Y_2Y_nu(P)\\
\vdots\\
\frac{Y_nX_1+X_1Y_n}{2}u(P),&\dots,\frac{Y_nX_n+X_nY_n}{2}u(P),&Y_nY_1u(P),&\dots,\dots,Y_n^2u(P)
\end{array}\right].
$$
Now we neeed the Taylor's formula adapted to our framework. 
Let $u$ be a smooth function   defined in an open set $\Omega\subset\mathbb{H}^n$ neighborhood of $0.$ Let  $\varepsilon_0$ be  a positive small number such that  
for every $0\leq s\leq 1,$ $\delta_{s}(P)\in \Omega.$  In such a way  the function
$$
g(s)=u(\delta_{s}(P))=u(s x,s y,s^2 t)
$$
is well defined for every $s\in[0,1].$ 
By the classical Taylor's formula centered at $0$, we get
$$
g(s)=g(0)+g'(0)s+\frac{1}{2}g''(0)s^2+\frac{1}{6}g'''(0)\bar{s},
$$
where $s\in(0,1).$
In particular $g(0)=u(0),$
so that
\begin{equation}
\begin{split}
g'(s)&=\sum_{i=1}^n(\partial_{x_i}u(\delta_{s}(P))x_i+\partial_{y_i}u(\delta_{s}(P))y_i)+2st\partial_tu(\delta_{s}(P))\\
&=\sum_{i=1}^n(\partial_{x_i}u(\delta_{s}(P))x_i+\partial_{y_i}u(\delta_{s}(P))y_i+2x_iy_i\partial_tu(\delta_{s}(P))-2x_iy_i\partial_tu(\delta_{s}(P)))+2st\partial_tu(\delta_{s}(P))\\
&=\sum_{i=1}^n(X_iu(\delta_{s}(P))x_i+Y_iu(\delta_{s}(P))y_i)+2st\partial_tu(\delta_{s}(P))\\
&=\langle\nabla_{\mathbb{H}^n}u(\delta_{s}(P)),(x,y)\rangle+2st\partial_tu(\delta_{s}(P)),
\end{split}
\end{equation}
and
\begin{equation}
\begin{split}
g''(s)
&=\sum_{i=1,j=1}^n(X_jX_iu(\delta_{s}(P))x_ix_j+Y_jX_iu(\delta_{s}(P))x_iy_j+X_jY_iu(\delta_{s}(P))y_ix_j+Y_jY_iu(\delta_{s}(P))y_iy_j\\
&+2st(\partial_tX_iu(\delta_{s}(P))+\partial_tY_iu(\delta_{s}(P))))+2t\partial_tu(\delta_{s}(P))+4s^2t\partial_{tt}u(\delta_{s}(P))\\
&=\langle D_{\mathbb{H}^n}^{2*}u(\delta_{s}(P))(x,y),(x,y)\rangle\\
&+2st(\partial_tX_iu(\delta_{s}(P))+\partial_tY_iu(\delta_{s}(P)))+2t\partial_tu(\delta_{s}(P))+4s^2t\partial_{tt}u(\delta_{s}(P)),
\end{split}
\end{equation}
and
\begin{equation*}
\begin{split}
g''(0)=\langle D_{\mathbb{H}^n}^{2*}u(0)(x,y),(x,y)\rangle+2t\partial_tu(0).
\end{split}
\end{equation*}
and by analogous calculation, for $\|(x,y,t)\|^2\leq\varepsilon_0,$ it results:
$$
|g'''(0)\bar{s}|\leq C\|(x,y,t)\|^2,
$$
where
$$
\|(x,y,t)\|=\sqrt[4]{(|x|^2+|y|^2)^2+t^2}.
$$
Hence by taking $s=1$ we get:
\begin{equation}\label{svilu}
u(x,y,t)=u(0)+\langle\nabla_{\mathbb{H}^n}u(0),(x,y)\rangle+\frac{1}{2}
\left(\langle D_{\mathbb{H}^n}^{2*}u(0)(x,y),(x,y)\rangle+2t\partial_tu(0)\right)+o(\|(x,y,t)\|^2).
\end{equation}
If $P\in\mathbb{H}^n$ and $$V\in\mathfrak{g}=\mbox{span(Lie)}\{X_{i},Y_j,[X_i,Y_j]:\:\:i,j=1,\dots,n\}$$ we set
$\vartheta_{(V,p)}(s):=\mbox{exp}[sV](P)\, (s\in\mathbb{R})$, i.e.
$\vartheta_{(V,p)}$ denotes the integral curve of $V$ starting
from $P$ and it turns out to be a {\it{1-parameter sub-group}} of $\mathbb{H}^n$. The
{\it{Lie group exponential map}} is defined by
$$\mbox{exp}:\mathfrak{g}\longmapsto\mathbb{H}^n,\quad\mbox{exp}(V):=\mbox{exp}[V](1).$$
The map $\mbox{exp}$ is an analytic diffeomorphism between $\mathfrak{g}$ and $\mathbb{H}^n.$
 One has
$$\vartheta_{(V,P)}(s)=P\circ\mbox{exp}(s
V)\quad\forall\,\,s\in\mathbb{R}.$$

In particular we remark that if $U\in H\mathbb{H}^n(P),$ then
$$
\vartheta_{(U,P)}(t)=P\circ\mbox{exp}(s
U),
$$
is horizontal.

Indeed, we say that a path $\phi:[-\tau,\tau]\to \mathbb{H}^n$ in the Heisenberg group is horizontal if $\phi'(s)\in H\mathbb{H}^n(\phi(s))$ for almost all $s\in [-\tau,\tau].$

Concerning the natural Sobolev spaces to consider in the Heisenberg group $\mathbb{H}^n$, we refer to the literature, see for instance \cite{GN}. Here, we only recall that 
$$
\mathcal{L}^{1,2}(\Omega):=\{f\in L^{2}(\Omega): X_if,\:\:Y_if\in L^{2}(\Omega),\:\:i=1,\dots, n\}
$$ 
is a Hilbert space with the norm
$$
|f|_{\mathcal{L}^{1,2}(\Omega)}=\left(\int_{\Omega}(\sum_{i}^n(X_if)^2+(Y_if)^2)+|f|^2dx\right)^{\frac{1}{2}}.
$$
Moreover,
$$
 H_{\mathbb{H}^n}^1(\Omega)=\overline{C^{\infty}(\Omega)\cap \mathcal{L}^{1,2}(\Omega)}^{|\cdot |_{\mathcal{L}^{1,2}(\Omega)}}.
$$
$$
 H_{\mathbb{H}^n,0}^1(\Omega)=\overline{C^{\infty}_0(\Omega)}^{|\cdot |_{\mathcal{L}^{1,2}(\Omega)}}.
$$
Of course, on the Sobolev-Poincar\'e inequalites there exists a wide literature, see e.g. \cite{Jerison_P}, \cite{FraLan}, \cite{CaDaGa}, \cite{Lu}. However here we shall recall only the following one in the Heisenberg group for every $u\in H_{\mathbb{H}^n, 0}^1(B_r)$
$$
\int_{B_r} |u(x)|dx \leq Cr \int_{B_r} |\nabla_{\mathbb{H}^n}u(x)|dx,
$$
see also \cite{GN} for isoperimetric and Sobolev inequalities in more general situations.

In general this presentation makes sense also for more stratificated non-commutative structures: the Carnot groups. In fact,
let $(\mathbb{G},\circ)$ be a group and  
there exist $\{\mathfrak{g}_{i}\}_{1\leq i\leq m},$ $m\in\mathbb{N},$ $m\leq N\in \mathbb{N},$ vector spaces such that:
$$
\mathfrak{g}_1\bigoplus \mathfrak{g}_2\bigoplus\dots\bigoplus \mathfrak{g}_m=\mathfrak{g}\equiv \mathbb{R}^N\equiv\mathbb{G}
$$
and
$$
[\mathfrak{g}_1,\mathfrak{g}_1]=\mathfrak{g}_2,\quad [\mathfrak{g}_1,\mathfrak{g}_2]=\mathfrak{g}_3,\quad\dots, [\mathfrak{g}_1,\mathfrak{g}_{m-1}]={g}_m,
$$
where
$$
[\mathfrak{g}_1,\mathfrak{g}_{m}]={0}.
$$

In this case we say that $\mathbb{G}$ is a stratified Carnot group of step $m.$

Moreover for every
$$
x\in \mathbb{G}\equiv\mathbb{R}^N=\mathbb{R}^{k_1}\times\dots\mathbb{R}^{k_m},\quad \sum_{j=1}^{m}k_i=N
$$
and for every $\lambda>0$ is defined the anisotropic dilation:
$$
\delta_{\lambda}(x)=(\lambda x^{(1)},\lambda^2 x^{(2)},\dots,\lambda^mx^{(m)}),\quad \mbox{where}\quad x^{(j)}\in\mathbb{R}^{k_j},\quad j=1,\dots,m
$$
such that, if  $Z_1,\dots,Z_{k_1}\in g_1$ are left invariant vector fields and $Z_j(0)=\frac{\partial}{\partial x_j}_{|x=0},$ $j=1,\dots,k_1,$ then
$$
\mbox{rank}(\mbox{Lie}\{Z_1,\dots,Z_{k_1}\})(x)=N, \quad (\mbox{H\"ormander condition})
$$
for every $x\in \mathbb{R}^N\equiv\mathbb{G}.$
 Let us consider the sublaplacian on the stratified Carnot group $\mathbb{G}$ given by
 $$
 \Delta_{\mathbb{G}}=\sum_{j=1}^{k_1}X_j^2.
 $$
 In particular there exists a $ N\times k_1$ matrix $\sigma$ such that $\sigma\cdot\sigma^T$ is a $N\times N$ matrix such that
 \begin{equation}\label{rap}
 \mbox{div}(\sigma\cdot\sigma^T\nabla \cdot)= \Delta_{\mathbb{G}}.
 \end{equation} 
 Moreover,
 $$
 \sigma^T\nabla u =\sum_{j=1}^{k_1}X_juX_j\equiv\nabla_{g_1}u,
 $$
 the so called horizontal gradient of $u.$
Hence
 $$
 A=\sigma\cdot\sigma^T.
 $$
 The Heisenberg group $\mathbb{H}^1$ is an example of Carnot group of step $2.$ In fact
 $$
 \mathfrak{g}_1=\mbox{span}\{X,Y\},\quad \mathfrak{g}_2=\mbox{span}\{[X,Y]\}, \quad [\mathfrak{g}_1,\mathfrak{g}_2]=\{0\}
 $$
and the Lie algebra of the Heisenberg group is obtained as 
$$
\mathfrak{g}=\mbox{span(Lie)}\{X,Y,[X,Y]\}=\mathfrak{g}_1\bigoplus\mathfrak{g}_2.
$$
Exploiting the cited representation of the sublaplacians \eqref{rap}, it results that
$$
\langle A\nabla u,\nabla u\rangle_{\mathbb{R}^N}=\langle \sigma\cdot\sigma^T\nabla u,\nabla u\rangle_{\mathbb{R}^N}=\langle\sigma^T\nabla u,\sigma^T\nabla u\rangle_{\mathbb{R}^{k_1}}\equiv \langle\nabla_{\mathbb{G}} u,\nabla_{\mathbb{G}} u\rangle_{\mathbb{R}^{k_1}},
$$
where $\nabla_{\mathbb{G}} u\equiv\sigma^T\nabla u=\sum_{j=1}^{k_1}X_juX_j$ is the horizontal gradient in the Carnot group $\mathbb{G}.$ Thus the definition by completion of the Sobolev spaces with respect to the norm
$$
||u||_{\mathbb{H}_{1,2}(\mathbb{G})}:=\sqrt{\int_{\Omega}\langle A\nabla u,\nabla u\rangle_{\mathbb{R}^N}+\int_{\Omega}u^2}\equiv\sqrt{\int_{\Omega}\langle\nabla_{\mathbb{G}} u,\nabla_{\mathbb{G}} u\rangle_{\mathbb{R}^{k_1}}+\int_{\Omega}u^2}
$$
is the same.

We spend few words about the Carnot-Charath\'eodory distance. To do this goal, we recall, see e.g. \cite{BLU}, that if $\{X_1,\dots X_N\}$  are vector fields in $\mathbb{R}^n,$ a piecewise regular path $\eta:[0,T]\to\mathbb{R}^n$ is said subunit, with respect to the family  $\{X_1,\dots X_N\},$ if for every $\xi\in \mathbb{R}^n$
$$
\langle \eta'(t),\xi\rangle^2\leq \sum_{j=1}^N\langle X_j(\eta(t)),\xi\rangle^2 ,\quad \mbox{for a.e.}\:\:t\in [0,T].
$$
Let us denote by $\mathcal{S}:=\mathcal{S}(\{X_1,\dots X_N\})$ the set of all the subunit paths.
\begin{prop}[Chow-Rashevsky] Let $\mathbb{G}=(\mathbb{R}^n,\circ, \delta_\lambda)$ be a Carnot group with the Lie algebra $\mathfrak{g}$ and let $\{X_1,\dots X_N\}$ be a family of vector fields in $\mathbb{R}^n.$
If 
$$
\mathfrak{g}=\mbox{Lie}\{X_1,\dots X_N\},
$$
then for every $x,y\in \mathbb{R}^n$ there exists $\eta\in \mathcal{S}$ such that $\eta(0)=x,$ $\eta(T)=y,$ moreover
$$
d_{CC}(x,y):=\inf\{T>0:\quad \mbox{there exists}\:\:\eta:[0,T]\to \mathbb{R}^n,\:\:\eta\in \mathcal{S},\:\: \eta(0)=x,\:\: \eta(T)=y\}
$$
is a distance called the Carnot-Charath\`eodory distance associated with $\{X_1,\dots X_N\}$ and $d_{CC}(\cdot,0)$ is a homogeneous norm on $\mathbb{G}.$ 
\end{prop}
In the case of the Heisenberg group, there exist positive constants $C_1,C_2>0$ such that for every $P\in \mathbb{H}^n$
$$
C_1\|P\|_{\mathbb{H}^n}\leq d_{CC}(P,0)\leq C_2\|P\|_{\mathbb{H}^n}.
$$
The same equivalence may be extended to Carnot groups, simply by considering the right homogeneous norm versus the Carnot-Charath\`eodory distance in the considered group $\mathbb{G}$. In addition,  a strong maximum principle holds, see \cite{Bony}, even if $\Delta_{\mathbb{G}}$ is a degenerate operator.
\begin{prop}
Let $u$ be  such that $\Delta_{\mathbb{G}}u\geq 0$ in $\Omega\subset \mathbb{G}$ is an open set and $\mathbb{G}$ is a group whose Lie algebra $\mathfrak{g}$ satisfies the H\"ormander condition. Then the supremum of $u$ can not be realized in $\Omega$ unless $u$ is constant.
\end{prop}
\section{The Bernoulli functional in the Heisenberg group}
In this section, following the scheme of \cite{ACF} we make some computations in the Heisenberg group $\mathbb{H}^n,$ but using the same arguments, the final results apply also to Carnot groups. In particular, here we recall that local minima of our functionals are globally continuous.
Let  
$$
J_{\mathbb{H}^n}(v)=\int_{\Omega}\left(\mid\nabla_{\mathbb{H}^n}v\mid^2+q^2(x)\lambda^2(v)+2fv\right)dx,\quad v\in K
$$
be the functional that we will study, where
$q^2(x)\not=0,$
\begin{equation}
\lambda(v)=\left\{
\begin{array}{l}
\lambda_1^2,\quad\mbox{if}\quad v<0,\\
\lambda_2^2,\quad\mbox{if}\quad v>0,
\end{array}
\right.
\end{equation}
and $\lambda^2(v)$ is lower semicontinuous at $v=0;$ it is assumed that $\lambda_i^2>0$ and $\Lambda=\lambda_1^2-\lambda_2^2\not=0.$
Here is
$$
K=\{v\in L_{\mbox{loc}}^1(\Omega):\:\:\nabla_{\mathbb{H}^n}v\in L^2(\Omega),\:\:v=u^0\:\: \text{on}\:\:S\subset\partial\Omega\}
$$
and $\Omega\subset \mathbb{R}^n $ is a domain.

There exists a unique solution to the following Dirichlet problem 
\begin{equation}
\left\{
\begin{array}{l}
\Delta_{\mathbb{H}^n}v_R=0,\quad B_R(0)\\
v_R=u,\quad\partial B_R(0).
\end{array}
\right.
\end{equation}

If $u$ realises a minimum for $J_{\mathbb{H}^n},$ then for every ball $B_r\subset \Omega$ we get:
\begin{equation*}
\int_{B_r(0)}\left(\mid\nabla_{\mathbb{H}^n}u\mid^2+q^2(x)\lambda^2(u)+2fu\right)dx\leq \int_{B_r(0)}\left(\mid\nabla_{\mathbb{H}^n}v_r\mid^2+q^2(x)\lambda^2(v_r)+2fv_r\right)dx.
\end{equation*}
Hence by the Poincar\' e ineguality we get:
\begin{equation*}\begin{split}
&\int_{B_r(0)}\left(\mid\nabla_{\mathbb{H}^n}u\mid^2-\mid\nabla_{\mathbb{H}^n}v_r\mid^2\right)dx\leq \int_{B_r(0)}\left(q^2(x)\lambda^2(v_r)-q^2(x)\lambda^2(u)\right)+2f(v_r-u)dx\\
&\leq C(\lambda_1,\lambda_2,Q)r^Q+2\int_{B_r(0)}f(v_r-u)dx.
\end{split}
\end{equation*}
On the other hand,
\begin{equation*}\begin{split}
&\int_{B_r(0)}\langle \nabla_{\mathbb{H}^n}(u-v_r),\nabla_{\mathbb{H}^n}(u+v_r)\rangle dx= \int_{B_r(0)}\mid\nabla_{\mathbb{H}^n}(u-v_r)\mid^2+
2\int_{B_r(0)}\langle\nabla_{\mathbb{H}^n}(u-v_r),\nabla_{\mathbb{H}^n}v_r\rangle\\&=\int_{B_r(0)}|\nabla_{\mathbb{H}^n}(u-v_r)|^2-2\int_{B_r(0)}f(u-v_r)dx
\end{split}
\end{equation*}
and
\begin{equation*}\begin{split}
&\int_{B_r(0)}\langle \nabla_{\mathbb{H}^n}(u-v_r),\nabla_{\mathbb{H}^n}(u+v_r)\rangle dx=\int_{B_r(0)}\left(|\nabla_{\mathbb{H}^n}u|^2-|\nabla_{\mathbb{H}^n}v_r|^2\right)dx.
\end{split}
\end{equation*}
Hence

\begin{equation*}\begin{split}
 \int_{B_r(0)}|\nabla_{\mathbb{H}^n}(u-v_r)|^2=&\int_{B_r(0)}\left(|\nabla_{\mathbb{H}^n}u|^2-|\nabla_{\mathbb{H}^n}v_r|^2\right)dx+2\int_{B_r(0)}f(u-v_r) dx. \end{split}
\end{equation*}
That is, by H\"older inequality

\begin{equation*}\begin{split}
 \int_{B_r(0)}|\nabla_{\mathbb{H}^n}(u-v_r)|^2\leq&C(\lambda_1,\lambda_2,Q)r^Q+4\|f\|_{L^Q(B_r(0))}\Big ( \int_{B_r(0)}|(u-v_r)|^2\Big )^{1/2}r^{\frac{Q-2}{2}}\end{split}
\end{equation*}
and, recalling Sobolev-Poincar\' e inequality one more time, we get:

\begin{equation*}\begin{split}
 \int_{B_r(0)}|\nabla_{\mathbb{H}^n}(u-v_r)|^2\leq&C(\lambda_1,\lambda_2,Q)r^Q+c'\|f\|_{L^Q(B_r(0))}\Big ( \int_{B_r(0)}|\nabla_{\mathbb{H}^n}(u-v_r)|^2\Big)^{1/2}r^{\frac{Q}{2}}. \end{split}
\end{equation*}
Thus, 
applying Cauchy inequality we get  for $\varepsilon>0$
\begin{equation*}\begin{split}
 \int_{B_r(0)}|\nabla_{\mathbb{H}^n}(u-v_r)|^2\leq&C(\lambda_1,\lambda_2,Q)r^Q+\frac{c'}{2\varepsilon}\|f\|_{L^Q(B_r(0))}^2r^{Q}+\frac{c'\varepsilon}{2} \int_{B_r(0)}|\nabla_{\mathbb{H}^n}(u-v_r)|^2 \end{split}
\end{equation*}
that implies
\begin{equation*}\begin{split}
(1-\frac{c'\varepsilon}{2}) \int_{B_r(0)}|\nabla_{\mathbb{H}^n}(u-v_r)|^2\leq&C(\lambda_1,\lambda_2,Q,\bar{\varepsilon}, \|f\|_{L^Q(B_r(0))}) r^Q, \end{split}
\end{equation*}
where
$$
C(\lambda_1,\lambda_2,Q,\bar{\varepsilon}, \|f\|_{L^Q(B_r(0))}):=C(\lambda_1,\lambda_2,Q)+\frac{c'}{2\varepsilon}\|f\|_{L^Q(B_r(0))}^2.
$$
Thus, by fixing $\bar{\varepsilon}>0$ such that $1-\frac{\varepsilon}{2} c'>\frac{1}{2}$ we conclude that there exists a constant $\bar{C}:=\bar{C}(\lambda_1,\lambda_2,\bar{\varepsilon}, \|f\|_{L^Q(\Omega)},Q)$ such that:
\begin{equation*}\begin{split}
 \int_{B_r(0)}|\nabla_{\mathbb{H}^n}(u-v_r)|^2\leq&\bar{C}r^Q. \end{split}
\end{equation*}


As a consequence, in analogy with the Euclidean case, we can not expect  on $u$ more than a modulus of continuity ruled by the Carnot-Charath\'eodory distance like, see the argument used by \cite{ACF}, \cite{Morrey} and \cite{LDT}:
\begin{equation*}
|u(x)-u(y)|\leq Cd_{CC}(x,y)|\log\left (\frac{1}{d_{CC}(x,y)}\right)|,
\end{equation*}
for every $x,y\in K,$ $d_{CC}(x,y)<\frac{1}{2}.$
The existence of a global Lipschitz intrinsic modulus of continuity may be face having a monotonicity formula. In $\mathbb{H}^1,$ see some partial results obtained in \cite{FeFo2_ar} and \cite{FeFo2}.
\section{Variation domains solutions for non-negative matrix}\label{variational_x}
In this section we face the general case with variable coefficients.

Let us consider the functional
$$
\mathcal{E}_A(v)=\int_{\Omega}\left(\langle A(x)\nabla v,\nabla v\rangle +M^2(v,x)+2fv\right),
$$
where $\langle A(x)\xi,\xi\rangle\geq 0$ for every $x\in \Omega,$ for every $\xi\in \mathbb{R}^n$, and
$$
M(u,x)=q(x)(\lambda_+\chi_{\{u>0\}}+\lambda_{-}\chi_{\{u<0\}}),
$$
where $\lambda_+,\lambda_-$ are non-negative numbers and $q\not\equiv 0$ is a function.

We define $\tau_\varepsilon(x)=x+\varepsilon\varphi(x)$ where $\varphi\in C_0^\infty(\Omega,\mathbb{R}^n).$ Recalling Section \ref{keytool} we remark that $A$ might be one of the matrices that is associated with a sublaplacian.

\begin{lem}
Let $u\in K$ be a local minimum of $\mathcal{E}_A.$ Then $u$ satisfies $\mbox{div}(A(x)\nabla u(x))=f$ in $\Omega\setminus\{u=0\}$
\end{lem}
\begin{proof}
For every $\varphi\in C_{0}^{\infty}(\Omega\setminus\{u=0\})$ and for every $\varepsilon>0$ sufficiently small, then
\begin{equation}\begin{split}
\mathcal{E}_A(u+\varepsilon \varphi)&=\int_{\Omega}\langle A(x)\nabla u,\nabla u\rangle+ 2\varepsilon\int_{\Omega}\langle A(x)\nabla u,\nabla \varphi\rangle\\
&+\varepsilon^2\int_{\Omega}\langle A(x)\nabla \varphi,\nabla \varphi\rangle+\int_{\Omega}M^2(u+\varepsilon \varphi,x)+2\int_{\Omega}f(u+\varepsilon \varphi)\\
&=\mathcal{E}_A(u)+ 2\varepsilon\int_{\Omega}\langle A(x)\nabla u,\nabla \varphi\rangle+2\varepsilon\int_{\Omega}f\varphi +o(\varepsilon^2).
\end{split}
\end{equation}
As a consequence,
\begin{equation}\begin{split}
\frac{\mathcal{E}_A(u+\varepsilon \varphi)-\mathcal{E}_A(u)}{\varepsilon}= 2\left(\int_{\Omega}\langle A(x)\nabla u,\nabla \varphi\rangle+\int_{\Omega}f\varphi \right)+o(\varepsilon)
\end{split}
\end{equation}
and 
\begin{equation}\begin{split}
\lim_{\varepsilon\to 0^+}\frac{\mathcal{E}_A(u+\varepsilon \varphi)-\mathcal{E}_A(u)}{\varepsilon}= 2\left(\int_{\Omega}\langle A(x)\nabla u,\nabla \varphi\rangle+\int_{\Omega}f\varphi \right)=0,
\end{split}
\end{equation}
that is $\mbox{div}(A(x)\nabla u(x))=f$ in $\Omega\setminus\{u=0\}$ in the weak sense.
\end{proof}
\begin{teo}
Let $u$ be a local minimum of $\mathcal{E}_A$ and $\mbox{meas}_n(\{u=0\})=0.$ Then $u$ is a domain variation solution and for every $\varphi\in C_0(\Omega,\mathbb{R}^n)$
\begin{equation*}
\begin{split}
&0=\lim_{\epsilon\to 0}\int_{\partial\{-\epsilon<u\}}\langle \varphi,\nu\rangle(M^2- \langle A(x)\nabla u^+,\nabla u^+\rangle) d S+
\lim_{\delta\to 0}\int_{\partial\{u<\delta\}}\langle \varphi,\nu\rangle(M^2- \langle A(x)\nabla u^-,\nabla u^-\rangle) d S.
\end{split}
\end{equation*}
 \end{teo}
 \begin{proof}
Denoting by $u_\varepsilon$ the function such that $u_{\varepsilon}(\tau_{\varepsilon}x)=u(x)$ where $\tau_{\varepsilon}=x+\varepsilon \varphi,$ $\varphi\in C_{0}^{\infty}(\Omega,\mathbb{R}^n)$ and assuming that $A$ is smooth, we get:

\begin{equation*}
\begin{split}&J(u_\varepsilon)=\int_{\Omega} \left(\langle A(y)\nabla u_\varepsilon(y),\nabla u_\varepsilon(y)\rangle+M^2(u_\varepsilon(y), y)+2f(y)u_\varepsilon(y)\right)dy\\
&=\int _{\Omega}\left( \langle A(\tau_\varepsilon(x))\nabla u_\varepsilon(\tau_\varepsilon (x)),\nabla u_\varepsilon(\tau_\varepsilon (x))\rangle+M^2(u(\tau_\varepsilon(x)), \tau_\varepsilon(x))+2fu(\tau_\varepsilon(x))\right)|\mbox{det}J\tau_\varepsilon|dx
\end{split}
\end{equation*}
On the other hand, since
$$
J\tau_\varepsilon (x)=I+\varepsilon J\varphi,
$$
then
$$
\mbox{det}J\tau_\varepsilon=1+\varepsilon \mbox{Tr}(J\varphi)+o(\varepsilon),
$$
for $\varepsilon\to 0.$
Moreover,
$$
\nabla u(x)=\nabla (u_\varepsilon(\tau_\varepsilon (x))=\nabla u_\varepsilon(\tau_\varepsilon (x))J\tau_\varepsilon(x),
$$
hence 
$$
J\tau_\varepsilon(x)^{-1}\nabla u(x)=\nabla u_\varepsilon(\tau_\varepsilon (x)).
$$
Keeping in mind that 
$$
J\tau_\varepsilon(x)^{-1}=I-\varepsilon J\varphi+o(\varepsilon),
$$
we conclude that 
$$
J\tau_\varepsilon(x)^{-1}\nabla u(x)=(I-\varepsilon J\varphi+o(\varepsilon))\nabla u(x)=\nabla u(x)-\varepsilon J\varphi \nabla u(x)+o(\varepsilon)
$$
and since $A$ is smooth we get
$$
A(\tau_\varepsilon(x))=A(x)+\varepsilon JA(x)\varphi+o(\varepsilon).
$$
As a consequence,
\begin{equation*}
\begin{split}&\int_\Omega \Big( \langle (A(x)+\varepsilon JA(x)\varphi+o(\varepsilon))J\tau_\varepsilon(x)^{-1}\nabla u(x),J\tau_\varepsilon(x)^{-1}\nabla u(x)\rangle\\
&+ M^2(u(x), \tau_\varepsilon(x))+2f(\tau_\varepsilon (x))u(x)\Big)|\mbox{det}J\tau_\varepsilon|dx\\
&=\int_{\Omega} \left( \langle (A(x)J\tau_\varepsilon(x)^{-1}\nabla u(x),J\tau_\varepsilon(x)^{-1}\nabla u(x)\rangle+M^2(u(x), \tau_\varepsilon(x))+2f(\tau_\varepsilon (x))u(x)\right)|\mbox{det}J\tau_\varepsilon|dx\\
&+\varepsilon\int_{\Omega} \left( \langle (JA(x)\phi+o(\varepsilon))J\tau_\varepsilon(x)^{-1}\nabla u(x),J\tau_\varepsilon(x)^{-1}\nabla u(x)\rangle\right)|\mbox{det}J\tau_\varepsilon|dx\\
&=\int_{\Omega} \left( \langle (A(x)\nabla u(x),\nabla u(x)\rangle+M^2(u(x), \tau_\varepsilon(x))+2f(\tau_\varepsilon (x))u(x)\right)|\mbox{det}J\tau_\varepsilon|dx\\
-&2\varepsilon \int_{\Omega}  \langle A(x)\nabla u(x),J\varphi\nabla u(x)\rangle |\mbox{det}J\tau_\varepsilon|dx +\varepsilon\int_{\Omega}\langle JA(x)\varphi\nabla u,\nabla u\rangle|\mbox{det}J\tau_\varepsilon|dx.
\end{split}
\end{equation*}
Hence
\begin{equation*}
\begin{split}
\frac{dJ(u_\varepsilon)}{d\varepsilon}_{|\varepsilon=0}&=\int_{\Omega} \langle (A(x)\nabla u(x),\nabla u(x)\rangle+M^2(u(x), x)+2f(x)u(x)\rangle)\mbox{Tr}(J\varphi)dx\\
&-2\int_{\Omega}\langle A\nabla u,J\varphi\nabla u\rangle dx +\int_{\Omega}\langle JA\varphi\nabla u,\nabla u\rangle dx +\int_{\Omega} \left(\langle\nabla_x M^2(u(x), x),\varphi\rangle+2\langle\nabla f(x),\varphi\rangle u\right)dx\\
&=\int_{\Omega} \langle (A(x)\nabla u(x),\nabla u(x)\rangle+M^2(u(x), x)\rangle)\mbox{Tr}(J\varphi)\\
&-2\int_{\Omega}\langle A\nabla u,J\varphi\nabla u\rangle +\int_{\Omega}\langle JA\varphi\nabla u,\nabla u\rangle dx+\int_{\Omega} \langle\nabla_x M^2(u(x), x),\varphi\rangle dx-2\int_{\Omega} f(x)\langle\varphi,\nabla u\rangle dx\\
&=\int_{\Omega}\mbox{div}\left(\left( \langle A(x)\nabla u(x),\nabla u(x)\rangle dx +M^2(u,x)\right)\varphi-2\langle\varphi,\nabla u\rangle A\nabla u\right)dx.
\end{split}
\end{equation*}
Since $u$ is a local minimum, then 
\begin{equation*}
\begin{split}
-\frac{dJ(u(x+\varepsilon\varphi(x)))}{d\varepsilon}_{|\varepsilon=0}=\frac{dJ(u_\varepsilon)}{d\varepsilon}_{|\varepsilon=0}=0,
\end{split}
\end{equation*}
that is $u$ is a domain variation solution.
Hence, forevery $\varphi\in C_0^{1}(\Omega,\mathbb{R}^n)$ we have:
\begin{equation*}
\begin{split}
&\frac{dJ(u_\varepsilon)}{d\varepsilon}_{|\varepsilon=0}\\
&=0=\int_{\Omega} \langle (A(x)\nabla u(x),\nabla u(x)\rangle+M^2(u(x), x)\rangle)\mbox{div}\varphi
-2\int_{\Omega}\langle A\nabla u,J\phi\nabla u\rangle \\
&+\int_{\Omega}\langle JA\varphi\nabla u,\nabla u\rangle+\int_{\Omega} \langle\nabla_x M^2(u(x), x),\varphi\rangle-2\int_{\Omega} f(x)\langle\varphi,\nabla u\rangle.
\end{split}
\end{equation*}
Now, let us consider now $\Omega= \{x\in \Omega:\:\:u<-\epsilon\}\cup  \{x\in\Omega:\:\:u>\delta\}\cup\{x\in \Omega:\:\:-\epsilon\leq u\leq \delta\},$ where $\epsilon, \delta>0.$ Then,  integrating by parts and denoting $\Omega_{\epsilon,\delta}(u)=\{x\in \Omega:\:\:-\epsilon\leq u\leq \delta\}$  as well as
\begin{equation*}
\begin{split}
&R_{\epsilon,\delta}:=\int_{\Omega_{\epsilon,\delta}(u)} \langle (A(x)\nabla u(x),\nabla u(x)\rangle+M^2(u(x), x)\rangle)\mbox{div}\varphi
-2\int_{\Omega_{\epsilon,\delta}(u)}\langle A\nabla u,J\phi\nabla u\rangle \\
&+\int_{\Omega_{\epsilon,\delta}(u)}\langle JA\varphi\nabla u,\nabla u\rangle+\int_{\Omega_{\epsilon,\delta}(u)} \langle\nabla_x M^2(u(x), x),\varphi\rangle-2\int_{\Omega_{\epsilon,\delta}(u)} f(x)\langle\varphi,\nabla u\rangle,
\end{split}
\end{equation*}

we get:
\begin{equation*}
\begin{split}
&0=-\int_{\Omega\cap\{u<-\epsilon\}} \langle\nabla\langle (A(x)\nabla u(x),\nabla u(x)\rangle+M^2(u(x), x)\rangle),\varphi\rangle  dx\\
&+\int_{\partial\{u<-\epsilon\}}\langle A(x)\nabla u(x),\nabla u(x)\rangle+M^2(u(x), x)\langle \varphi,\nu\rangle d\sigma\\
&-\int_{\Omega\cap \{u>\delta\}} \langle\nabla\langle (A(x)\nabla u(x),\nabla u(x)\rangle+M^2(u(x), x)\rangle),\varphi\rangle  dx\\
&+\int_{\partial\{u>\delta\}}\langle A(x)\nabla u(x),\nabla u(x)\rangle+M^2(u(x), x)\langle \varphi,\nu\rangle d\sigma
-2\int_{\Omega\cap (\{u>\delta\}\cup \{u<-\epsilon\})}\langle A\nabla u,J\phi\nabla u\rangle dx\\
&+\int_{\Omega\cap (\{u>\delta\}\cup \{u<-\epsilon\})}\langle JA\varphi\nabla u,\nabla u\rangle dx+\int_{\Omega\cap (\{u>\delta\}\cup \{u<-\epsilon\})} \langle\nabla_x M^2(u(x), x),\varphi\rangle dx\\
&-2\int_{\Omega\cap (\{u>\delta\}\cup \{u<-\epsilon\})} f(x)\langle\varphi,\nabla u\rangle dx+R_{\epsilon,\delta}.
\end{split}
\end{equation*}
Thus, by recalling that $u$ satisfies $\mbox{div}(A\nabla u)= f(x)$ in $\Omega\setminus \{u=0\}$  we get, denoting $u^+:=\sup\{u,0\}$ and $u^-:=\sup\{-u,0\}$:
\begin{equation}\label{calcul}
\begin{split}
0=&\lim_{\epsilon\to 0}\int_{\partial\{-\epsilon<u\}}\langle \varphi,\nu\rangle(\langle A(x)\nabla u^+,\nabla u^+\rangle+M^2)dS+
\lim_{\delta\to 0}\int_{\partial\{u<\delta\}}\langle \varphi,\nu\rangle(\langle A(x)\nabla u^-,\nabla u^-\rangle+M^2)dS \\
&-2(\lim_{\epsilon\to 0}\int_{\partial\{-\epsilon<u\}}\langle \varphi,\nu\rangle\langle A(x)\nabla u^+,\nabla u^+\rangle dS+
\lim_{\delta\to 0}\int_{\partial\{u<\delta\}}\langle \varphi,\nu\rangle\langle A(x)\nabla u^-,\nabla u^-\rangle dS),
\end{split}
\end{equation}
because by hypothesis $\mbox{meas}_n(\{u=0\})=0$ so that $\lim_{\epsilon,\delta\to 0}R_{\epsilon,\delta}=0.$

Finally \eqref{calcul} leads to
\begin{equation*}
\begin{split}
&0=\lim_{\epsilon\to 0}\int_{\partial\{-\epsilon<u\}}\langle \varphi,\nu\rangle(M^2- \langle A(x)\nabla u^+,\nabla u^+\rangle) d S+
\lim_{\delta\to 0}\int_{\partial\{u<\delta\}}\langle \varphi,\nu\rangle(M^2- \langle A(x)\nabla u^-,\nabla u^-\rangle) d S.
\end{split}
\end{equation*}
\end{proof}
In conclusion we have obtained, whenever $\mbox{meas}_{n}\{u=0\}=0,$ that:
\begin{equation}\label{two_phase_matrix}
\begin{cases}
\mbox{div}(A(x)\nabla u)=f& \mbox{in }\Omega^+(u):= \{x\in \Omega:\hspace{0.1cm} u(x)>0\}\\
\mbox{div}(A(x)\nabla u)=f& \mbox{in }\Omega^-(u):=\mbox{Int}(\{x\in \Omega:\hspace{0.1cm} u(x)\leq 0\})\\
\langle A\nabla u^+ u^+\rangle-\langle A\nabla u^- u^-\rangle= q(x)\Lambda&\mbox{on }\mathcal{F}(u):=\partial \Omega^+(u)\cap \Omega.
\end{cases}
\end{equation}
where $\Lambda:=\lambda_+^2-\lambda_-^2.$ In the case of the Heisenberg group this reads as follows (see Section \ref{groupCa} for the details and further generalizations):
\begin{equation}\label{two_phase_Heisenberg_2}
\begin{cases}
\Delta_{\mathbb{H}^n} u=f& \mbox{in }\Omega^+(u):= \{x\in \Omega:\hspace{0.1cm} u(x)>0\}\\
\Delta_{\mathbb{H}^n} u=f& \mbox{in }\Omega^-(u):=\mbox{Int}(\{x\in \Omega:\hspace{0.1cm} u(x)\leq 0\})\\
|\nabla_{\mathbb{H}^n} u^+|^2-|\nabla_{\mathbb{H}^n} u^-|^2= q(x)\Lambda&\mbox{on }\mathcal{F}(u):=\partial \Omega^+(u)\cap \Omega.
\end{cases}
\end{equation}

\section{Some comments about Heisenberg group and Carnot groups}\label{groupCa}
We compute $\langle A(x)\nabla u,\nabla u\rangle$ assuming that 
$$
A=\left[
\begin{array}{lll}
1,&0,&2y\\
0,&1,&-2x\\
2y,&-2x,&,4(x^2+y^2)
\end{array}
\right].
$$
Then
\begin{equation*}
\begin{split}
&\langle A\nabla u,\nabla u\rangle=\left[
\begin{array}{l}
Xu\\
Yu\\
2y\frac{\partial u}{\partial x}-2x\frac{\partial u}{\partial y}+4\frac{\partial u}{\partial t}(x^2+y^2)
\end{array}
\right]\cdot \nabla u\\
&=Xu\frac{\partial u}{\partial x}+Yu\frac{\partial u}{\partial y}+(2y\frac{\partial u}{\partial x}-2x\frac{\partial u}{\partial y})\frac{\partial u}{\partial t}+4(\frac{\partial u}{\partial t})^2(x^2+y^2)\\
&=(Xu)^2-2yXu\frac{\partial u}{\partial t}+(Yu)^2+2xYu\frac{\partial u}{\partial t}\\
&+(2y\frac{\partial u}{\partial x}-2x\frac{\partial u}{\partial y})\frac{\partial u}{\partial t}+4(\frac{\partial u}{\partial t})^2(x^2+y^2)=(Xu)^2+(Yu)^2\\
&=\vert\nabla_{\mathbb{H}^1} u\vert^2=\langle\nabla_{\mathbb{H}^1}u, \nabla_{\mathbb{H}^1}u\rangle_{\mathbb{H}^1}
\end{split}
\end{equation*}

Notice that
$$
\mbox{div}(A(x)\nabla u(x))=X^2u+Y^2u=\Delta_{\mathbb{H}^1}u=\mbox{div}_{\mathbb{H}^1}(\nabla_{\mathbb{H}^1}u)=X(Xu)+Y(Yu).
$$

An other example, may be give for the Engel group. In this case we have:

$$
g_1\bigoplus g_2 \bigoplus g_3,
$$

where
$$
g_1=\mbox{span}\{X_1,X_2\},\quad g_2=\mbox{span}\{X_3\},\quad g_3=\mbox{span}\{X_4\},
$$

$$
[X_1,X_2]=X_3,\quad [X_1, X_3]=X_4,
$$

$$
X_1=\frac{\partial}{\partial x_1}-x_2\frac{\partial}{\partial x_3}-x_3\frac{\partial}{\partial x_4},\quad X_2=\frac{\partial}{\partial x_2},\quad X_3=\frac{\partial}{\partial x_3},\quad X_4=\frac{\partial}{\partial x_4}
$$

$$
x y=(x_1+y_1,x_2+y_2,x_3+y_3-y_1x_2, x_4+y_4+\frac{1}{2}y_1^2x_2-y_1x_3).
$$
Moreover
$$
\left[\begin{array}{ll}
1,&0\\
0,&1\\
-x_2,&0\\
-x_3,&0
\end{array}\right]\left[\begin{array}{llll}
1,&0,&-x_2,&-x_3\\
0,&1,&0,&0
\end{array}\right]=\left[\begin{array}{llll}
1,&0,&-x_2,&-x_3\\
0,&1,&0,&0\\
-x_2,&0,&x_2^2,&x_2x_3\\
-x_3,&0,&x_2x_3,&x_3^2
\end{array}\right]
$$
In this case:
$$
\Delta_{E}=X_1^2+X_2^2.
$$

We can generalize this remark. Indeed, see Section 1.5-(A3) in \cite{BLU}, it is well known that every sublaplacian $\Delta_{\mathbb{G}}=\sum_{i=1}^{n_1}Z_i^2$ on a group $\mathbb{G}$ can be written
in divergence form as:
$$
\Delta_{\mathbb{G}}=\mbox{div}(A(x)\nabla),
$$
where
\begin{equation}\label{squareroot_group}
A=\sigma(x)\sigma^T(x)
\end{equation}
and $\sigma$ is the $n\times n_1$ matrix whose columns are given by the coefficients of the vector fields  $Z_1,\dots,Z_{n_1}.$

We conclude that the two-phase problems for Carnot sublaplacians have to satisfy, whenever $\mbox{meas}_{\mathbb{G}}(\{u=0\})=0,$ the following condition on the free boundary
\begin{equation*}
\begin{split}
&0=\lim_{\epsilon\to 0}\int_{\{-\epsilon<u\}}\langle \phi,\nu\rangle(M^2- \mid \nabla_{\mathbb{G}} u^+\mid^2) d S+\lim_{\delta\to 0}\int_{\{u<\delta\}}\langle \phi,\nu\rangle(M^2- \mid \nabla_{\mathbb{G}} u^-\mid^2) d S,
\end{split}
\end{equation*}
where $\mid \nabla_{\mathbb{G}} u\mid^2=\sum_{i=1}^{n_1}(Z_iu)^2.$ Then
\begin{equation}\label{two_phase_Carnot}
\left\{\begin{array}{lr}
\Delta_{\mathbb{G}} u=f,& \mbox{in }\Omega^+(u):= \{x\in \Omega:\hspace{0.1cm} u(x)>0\},\\
\Delta_{\mathbb{G}} u=f,& \mbox{in }\Omega^-(u):=\mbox{Int}(\{x\in \Omega:\hspace{0.1cm} u(x)\leq 0\}),\\
|\nabla_{\mathbb{G}} u^+|^2-|\nabla_{\mathbb{G}} u^-|^2=q(x)(\lambda_+^2-\lambda_-^2):=q(x)\Lambda&\mbox{on }\mathcal{F}(u):=\partial \Omega^+(u)\cap \Omega,
\end{array}
\right.
\end{equation}
where, whatever the function $u$ is sufficiently smooth, it results:
$$
|\nabla_{\mathbb{G}} u|^2=\langle A(x)\nabla u,\nabla u\rangle=\langle \sigma^T\nabla u,\sigma^T\nabla u\rangle_{\mathbb{R}^{n_1}}
$$
and
$$
\nabla_{\mathbb{G}} u(x):=\sigma^T(x)\nabla u(x)=\sum_{k=1}^{n_1}Z_ku(x)Z_k(x).
$$
In the case of $\mathbb{H}^1,$ the functions like $\alpha(ax+by)^+-\beta(ax+by)^-$, where $a^2+b^2>0,$ $a,b\in \mathbb{R}$ are fixed, as well as $\alpha,\beta\in\mathbb{R},$ $\alpha,\beta>0,$  satisfy the two-phase homogeneous problem 
\begin{equation}\label{two_phase_Heisenberg}
\left\{\begin{array}{lr}
\Delta_{\mathbb{H}^1} u=0,& \mbox{in }\Omega^+(u):= \{x\in \Omega:\hspace{0.1cm} u(x)>0\},\\
\Delta_{\mathbb{H}^1} u=0,& \mbox{in }\Omega^-(u):=\mbox{Int}(\{x\in \Omega:\hspace{0.1cm} u(x)\leq 0\}),\\
|\nabla_{\mathbb{H}^1} u^+|^2-|\nabla_{\mathbb{H}^1} u^-|^2=(a^2+b^2)(\alpha^2-\beta^2)&\mbox{on }\mathcal{F}(u):=\partial \Omega^+(u)\cap \Omega.
\end{array}
\right.
\end{equation}
In this case the free boundary $\mathcal{F}(u)$ is the set $\{(x,y,t)\in \mathbb{H}^1:\quad ax+by=0\}$ that does not have characteristic points.
\section{Nonlinear case: $p(x)-$Laplace operator}\label{pxlapvar}
We now argue considering
 the following functional

$$
J(u)=\int_{\Omega}\left(a(\mid \nabla u\mid, x) +M^2(u,x)+p(x)f(x)u\right)dx,
$$

where 
$$
M(u,x)=q(x)(\lambda_+\chi_{\{u>0\}}+\lambda_{-}\chi_{\{u<0\}})
$$
and $a$ is a function that we shall introduce in a while.

We define $\tau_\varepsilon(x)=x+\varepsilon\phi(x)$ where $\phi\in C_0^\infty(\Omega,\mathbb{R}^n).$ Then
denoting by $u_\varepsilon$ the function such that $u_{\varepsilon}(\tau_{\varepsilon}x)=u(x)$ and assuming that $A$ is smooth, we get:
\begin{equation*}
\begin{split}&J(u_\varepsilon)=\int_{\Omega} \left( a(\mid\nabla u_\varepsilon(y)\mid, y)+M^2(u_\varepsilon(y), y)+p(x)f(y)u_\varepsilon(y)\right)dy\\
&=\int_{\Omega} \left( a(\mid\nabla u_\varepsilon(\tau_\varepsilon (x))\mid, \tau_\varepsilon(x))+M^2(u(\tau_\varepsilon(x)), \tau_\varepsilon(x))+pf(\tau_\varepsilon(x))u(\tau_\varepsilon(x))\right)|\mbox{det}J\tau_\varepsilon|dx
\end{split}
\end{equation*}
On the other hand, following the same notation of the case described in Section \ref{variational_x}
we obtain:
\begin{equation*}
\begin{split}
&J(u_\varepsilon)
=\int_{\Omega} \left( a(\mid J\tau_\varepsilon(x)^{-1}\nabla u(x)\mid, \tau_\varepsilon(x))+M^2(u(x), \tau_\varepsilon(x))+pf(\tau_\varepsilon(x))u(x)\right)|\mbox{det}J\tau_\varepsilon| dx\\
&=\int_{\Omega} \Big( a(\mid \nabla u(x)-\varepsilon J\phi \nabla u(x)+o(\varepsilon)\mid, \tau_\varepsilon(x))+M^2(u(x),\tau_\varepsilon(x))+pf(\tau_\varepsilon(x))u(x)\Big)|\mbox{det}J\tau_\varepsilon| dx.
\end{split}
\end{equation*}
In the case when $a(b,c)=b^{p(c)},$  denoting
$$
\mathcal{E}_{p(x)}(u):=\int_{\Omega}\left(\mid \nabla u\mid^{p(x)}+M^2(u,x)+p(x)f(x)u\right)dx,
$$
we get from the Taylor expansion:
\begin{equation*}\begin{split}
&a(\mid \nabla u(x)-\varepsilon J\phi \nabla u(x)+o(\varepsilon)\mid, \tau_\varepsilon(x)\mid=\mid \nabla u(x)-\varepsilon J\phi \nabla u(x)+o(\varepsilon)\mid^{p(\tau_\varepsilon(x))}\\
&=\mid \nabla u(x)-\varepsilon J\phi \nabla u(x)+o(\varepsilon)\mid^{p(x)+\varepsilon\langle\nabla p(x),\varphi(x)\rangle+o(\varepsilon)}\\
&=\mid \nabla u(x)-\varepsilon J\phi \nabla u(x)+o(\varepsilon)\mid^{p(x)}\mid \nabla u(x)-\varepsilon J\phi \nabla u(x)+o(\varepsilon)\mid^{\varepsilon\langle\nabla p(x),\varphi(x)\rangle+o(\varepsilon)}\\
\end{split}
\end{equation*}
so that
\begin{equation}\begin{split}
&=(\mid \nabla u(x)|^2-2\varepsilon \langle J\phi \nabla u(x),\nabla u(x)+o(1)\rangle+o(\varepsilon))^{\frac{p(x)}{2}}\mid \nabla u(x)-\varepsilon J\phi \nabla u(x)+o(\varepsilon)\mid^{\varepsilon\langle\nabla p(x),\varphi(x)\rangle+o(\varepsilon)}\\
&=\left(\mid \nabla u(x)|^{p(x)}-\varepsilon p(x)\langle J\phi \nabla u(x),\nabla u(x)\rangle|\nabla u(x)|^{p(x)-2}+o(\varepsilon)\right)\\
&\times \exp\{\varepsilon (\langle\nabla p(x), \varphi(x)\rangle+o(1)) \log(\mid \nabla u(x)-\varepsilon J\phi \nabla u(x)+o(\varepsilon)\mid)\}\\
&=\left(\mid \nabla u(x)|^{p(x)}-\varepsilon p(x)\langle J\phi \nabla u(x),\nabla u(x)\rangle|\nabla u(x)|^{p(x)-2}+o(\varepsilon)\right)\\
&\times \exp\left(\varepsilon (\langle\nabla p(x), \varphi(x)\rangle+o(1)) \left(\log(\mid \nabla u(x)\mid)+\log(1-\varepsilon\langle J\phi \nabla u(x),\nabla u(x)\rangle+o(\varepsilon))\right)\right)\\
\end{split}
\end{equation}
that is
\begin{equation*}\begin{split}
&=\left(\mid \nabla u(x)|^{p(x)}-\varepsilon p(x)\langle J\phi \nabla u(x),\nabla u(x)\rangle|\nabla u(x)|^{p(x)-2}+o(\varepsilon)\right)\\
&\times \left(1+\varepsilon \langle\nabla p(x),\nabla \varphi(x)\rangle\log \mid \nabla u(x)\mid+o(\varepsilon))\right)\\
&=\mid \nabla u(x)|^{p(x)}+\varepsilon \left(\mid \nabla u(x)|^{p(x)} \langle\nabla p(x),\varphi(x)\rangle\log \mid \nabla u(x)| -p(x)\langle J\phi \nabla u(x),\nabla u(x)\rangle|\nabla u(x)|^{p(x)-2}\right)+o(\varepsilon).
\end{split}
\end{equation*}
As a consequence
\begin{equation*}\begin{split}
&\mathcal{E}_{p(x)}(u_\varepsilon)
=\int_{\Omega} \Big( \mid \nabla u(x)|^{p(x)}+\varepsilon \Big(\mid \nabla u(x)|^{p(x)} \langle\nabla p(x),\nabla \varphi(x)\rangle\log |\nabla u(x)|\\
& -p(x)\langle J\phi \nabla u(x),\nabla u(x)\rangle|\nabla u(x)|^{p(x)-2}\Big)+o(\varepsilon))\\
&+M^2(u(x), \tau_\varepsilon(x))+(p(x)+\varepsilon \langle\nabla p(x),\phi\rangle +o(\varepsilon)) f(\tau_\varepsilon(x))u(x)\Big)|\mbox{det}J\tau_\varepsilon| dx\\
\end{split}
\end{equation*}
so that
\begin{equation*}\begin{split}
&=\int_{\Omega} \Big( \mid \nabla u(x)|^{p(x)}+M^2(u(x), x)+p(x) f(x)u(x)\Big)\Big(1+\varepsilon \mbox{Tr}(J\phi)+o(\varepsilon)\Big) dx\\
&+\varepsilon\int_{\Omega} \Big(\mid \nabla u(x)|^{p(x)} \langle\nabla p(x), \varphi(x)\rangle\log \mid \nabla u(x)|\\
& -p(x)\langle J\phi \nabla u(x),\nabla u(x)\rangle|\nabla u(x)|^{p(x)-2}+u(x)\langle\nabla p(x),\phi\rangle f(x) \Big)\\
&\times\Big(1+\varepsilon \mbox{Tr}(J\phi)+o(\varepsilon)\Big) dx\\
&+\varepsilon\int_{\Omega} \left(p(x)\langle\nabla f(x),\phi\rangle u(x)+\langle \nabla M^2(u(x), x),\phi\rangle \right)\Big(1+\varepsilon \mbox{Tr}(J\phi)+o(\varepsilon)\Big) dx+o(\varepsilon)\\
\end{split}
\end{equation*}
from which follows:
\begin{equation*}\begin{split}
&=\mathcal{E}_{p(x)}(u)+\varepsilon\Big\{\int\Big(  \mid \nabla u(x)|^{p(x)}+ M^2(u(x), x)+p(x) f(x)u(x)\Big)\mbox{Tr}(J\phi)\\
&+\int_{\Omega}\Big(\mid \nabla u(x)|^{p(x)} \langle\nabla p(x), \varphi(x)\rangle\log \mid \nabla u(x)| -p(x)\langle J\phi \nabla u(x),\nabla u(x)\rangle|\nabla u(x)|^{p(x)-2}\\
&+\langle\nabla p(x),\phi\rangle u(x) f(x)+p(x)u(x)\langle\nabla f(x),\phi\rangle+\langle \nabla M^2(u(x), x),\phi\rangle \Big)\Big\}+o(\varepsilon).
\end{split}
\end{equation*}
Thus, recalling that $u$ is a minimum, we my conclude that
$$
\lim_{\varepsilon\to 0}\frac{\mathcal{E}_{p(x)}(u_\varepsilon)-\mathcal{E}_{p(x)}(u)}{\varepsilon}=0.
$$
Thus we deduce, recalling $\mbox{Tr}(J\phi)=\mbox{div}(\phi)$ that:
\begin{equation}\begin{split}
&0=\Big\{\int_{\Omega}\Big(  \mid \nabla u(x)|^{p(x)}+M^2(u(x), x)+p(x) f(x)u(x)\Big)\mbox{div}(\phi)\\
&+\int_{\Omega}\Big(\mid \nabla u(x)|^{p(x)} \langle\nabla p(x),\phi(x)\rangle\log \mid \nabla u(x)| -p(x)\langle J\phi \nabla u(x),\nabla u(x)\rangle|\nabla u(x)|^{p(x)-2}\\
&+\langle\nabla p(x),\phi\rangle u(x)f(x)+p(x)u(x)\langle\nabla f(x),\phi\rangle+\langle\nabla M^2(u(x), x),\phi\rangle \Big)\Big\},
\end{split}
\end{equation}
that is also
\begin{equation}\begin{split}
&0=\Big\{\int_{\Omega}\Big(  \mid \nabla u(x)|^{p(x)}+M^2(u(x), x)\Big)\mbox{div}(\phi)\\
&+\int_{\Omega}\Big(\mid \nabla u(x)|^{p(x)} \langle\nabla p(x), \phi(x)\rangle\log \mid \nabla u(x)| -p(x)\langle J\phi \nabla u(x),\nabla u(x)\rangle|\nabla u(x)|^{p(x)-2}\\
&+\langle\nabla M^2(u(x), x),\phi\rangle-f(x)p(x) \langle\nabla u,\phi\rangle\Big)\Big\}.
\end{split}
\end{equation}

Hence, integrating by parts and recalling that $\mbox{div}(\mid\nabla u\mid^{p(x)-2}\nabla u)=f$ in $\Omega\setminus F(u),$ we get,
considering $\Omega= \{x\in \Omega:\:\:u<-\epsilon\}\cup  \{x\in\Omega:\:\:u>\delta\}\cup\{x\in \Omega:\:\:-\epsilon\leq u\leq \delta\},$ where $\epsilon, \delta>0$, recalling that $\Omega_{\epsilon,\delta}(u)=\{x\in \Omega:\:\:-\epsilon\leq u\leq \delta\}$ and denoting by:
\begin{equation*}
\begin{split}
&R_{\epsilon,\delta}:=\int_{\Omega_{\epsilon,\delta}(u)} \Big(  \mid \nabla u(x)|^{p(x)}+M^2(u(x), x)\Big)\mbox{div}\varphi
-\int_{\Omega_{\epsilon,\delta}(u)}p(x)\langle\nabla J\phi \nabla u(x),\nabla u(x)\rangle \\
&+\int_{\Omega_{\epsilon,\delta}(u)}\mid \nabla u(x)|^{p(x)} \langle\nabla p(x),\phi(x)\rangle\log \mid \nabla u(x)|+\int_{\Omega_{\epsilon,\delta}(u)} \langle\nabla_x M^2(u(x), x),\varphi\rangle\\
&-\int_{\Omega_{\epsilon,\delta}(u)} p(x)f(x)\langle\varphi,\nabla u\rangle,
\end{split}
\end{equation*}

we get:

\begin{equation}\begin{split}
&0=\lim_{\epsilon\to 0,\delta\to 0}\Big\{\int_{\partial\{u<-\epsilon\}}\langle n,\phi\rangle \Big((1-p(x))\mid \nabla u(x)|^{p(x)}+M^2(u(x), x)\Big)dS\\
&+\int_{\partial\{u>\delta\}}(1-p(x))\langle n,\phi\rangle \Big(\mid \nabla u(x)|^{p(x)}+M^2(u(x), x)\Big)dS+R_{\epsilon,\delta}\Big\}\\
\end{split}
\end{equation}
that implies 
\begin{equation}\begin{split}
&0=\lim_{\epsilon\to 0,\delta\to 0}\Big\{\int_{\partial\{u<-\epsilon\}}\langle n,\phi\rangle \Big((1-p(x))\mid \nabla u(x)|^{p(x)}+M^2(u(x), x)\Big)dS\\
&+\int_{\partial\{u>\delta\}}\langle n,\phi\rangle \Big((1-p(x))\mid \nabla u(x)|^{p(x)}+M^2(u(x), x)\Big)dS\Big\}\\
\end{split}
\end{equation}
because we assumed that $\mbox{meas}_n\{u=0\}=0,$ so that $\lim_{\epsilon\to 0,\delta\to 0}R_{\epsilon,\delta}=0.$

As a consequence the natural pointwise condition on the free boundary $\{u=0\}$ is:
$$
(p(x)-1)\mid\nabla u^+\mid^{p(x)}-(p(x)-1)\mid\nabla u^-\mid^{p(x)}=q(x)(\lambda_+^2-\lambda_-^2).
$$
Usually previous condition is written as well as
$$
(u^+_n)^{p(x)}-(u^-_n)^{p(x)}=q(x)\frac{\lambda_+^2-\lambda_-^2}{p(x)-1},
$$
where $u^+_n$ and $u^-_n$ denote  the normal derivatives,  computed considering $n$ pointing inside to $\Omega^+(u)$ and $\Omega^-(u)$ respectively, at the points of the set $\{u=0\},$ of course whenever this fact makes sense. In fact for every $x\in \{u=0\},$ and such that $\nabla u(x)\not=0,$ we have:
$$
u_n(x)=\langle \nabla u(x),\frac{\nabla u(x)}{|\nabla u(x)|}\rangle=|\nabla u(x)|.
$$
In conclusion the two phase problem can be formulated in viscosity sense as:
 \begin{equation}\label{two_phase_p(x)_gen}
\left\{\begin{array}{lr}
\Delta_{p(x)} u=f,& \mbox{in }\Omega^+(u):= \{x\in \Omega:\hspace{0.1cm} u(x)>0\},\\
\Delta_{p(x)} u=f,& \mbox{in }\Omega^-(u):=\mbox{Int}(\{x\in \Omega:\hspace{0.1cm} u(x)\leq 0\}),\\
|\nabla u^+|^{p(x)}-|\nabla u^-|^{p(x)}=q(x)\frac{\Lambda}{p(x)-1},&\mbox{on }\mathcal{F}(u):=\partial \Omega^+(u)\cap \Omega,
\end{array}
\right.
\end{equation}
being $\Lambda:=\lambda_+^2-\lambda_-^2.$

\section{Conclusions}\label{conclusions}
Starting from the condition on the free boundary that we have obtained, in Carnot groups for the two phase problems, we ask to ourselves if a comparison result may work in this framework. Following the mentioned viscosity approach introduced in \cite{D} and \cite{DFS_apdes, DFS_japa, DFS_trans}, the first thing to prove seems to be the existence of a comparison result. From this point of view, it is natural to recall the properties arising from the Hopf maximum principle. About this subject in Carnot groups, we cite \cite{Birindelli_Cutri}, for a detailed study, for a discussion in the Heisenberg group, and \cite{Martino_Tralli} for a generalization to the Carnot groups. In fact in \cite{Birindelli_Cutri}, see Lemma 2.1, the authors remark that if a set $\Omega$ satisfies the inner intrinsic ball property, namely if $P_0\in\partial\Omega$ is such that there exists a Koranyi ball $B_R^{\mathbb{H}^1}(Q)\subset \Omega,$ such that $P_0=\partial B_R^{\mathbb{H}^1}(Q)\cap \partial \Omega,$ $u$ satisfies $\Delta_{\mathbb{H}^1}u(P)\geq 0$ and $u(P)>u(P_0)$  for every $P\in B_R^{\mathbb{H}^1}(P_0)\cap \Omega,$  then
$$
\lim_{h\to^+0}\frac{f(P_0)-f(P_0-th)}{t}<0,
$$
where $h$ denotes any exterior direction to $\partial \Omega$ at $P_0;$ moreover, in case if $\frac{\partial f(P_0)}{\partial h}$ exists, then $\frac{\partial f(P_0)}{\partial h}<0.$
In this order of ideas the right definition of a viscosity solution for \eqref{two_phase_Carnot_x}  may be the following one.

Unfortunately, if the contact point between the set and the ball is realized in a characteristic point, then $\frac{\partial f}{\partial h}=0$ at the characteristic points along all the horizontal admissible directions $h\in H\mathbb{H}^n,$ that is $\nabla_{\mathbb{H}^n}f=0$ at the characteristic points.

We denote by $\nu$ the intrinsic normal to $\mathcal{F}(v)$ at $x_0\in \mathcal{F}(v)$ and, as usual,  $v_\nu^+(x_0), \quad v_\nu^-(x_0)$ represent the horizontal derivatives with respect to the inner intrinsic normal $\nu$ to $\Omega^+(v)$ and to $\Omega^{-}(v)$ respectively. 

We are in position to state the definition of solution of a two-phase free boundary problem in a simpler case like \eqref{two_phase_Carnot_x} as follows:
\begin{defi}
We say that $u\in C(\Omega)$ is a solution to \eqref{two_phase_Carnot_x} if:
\begin{itemize}
\item[(i)] $\Delta_{\mathbb{G}} u=f$ in a viscosity sense in $\Omega^+(u)$ and $\Omega^-(u);$
\item[(ii)] let $x_0\in \mathcal{F}(u).$ For every function $v\in C(B_\varepsilon(x_0))$, $\varepsilon >0$ such that $v\in C^{2}(\overline{B^+(v)})\cap C^{2}(\overline{B^-(v)}),$ being $B:= B_\varepsilon(x_0)$  and $\mathcal{F}(v)\in C^2,$ if $v$ touches $u$ from below (resp. above) at $x_0\in \mathcal{F}(v),$ and $x_0$ is not characteristic for $\mathcal{F}(v),$  then 
$$
(v_\nu^+(x_0))^2-(v_\nu^-(x_0))^2\leq 1\quad (\mbox{resp.}\quad (v_\nu^+(x_0))^2-(v_\nu^-(x_0))^2\geq 1).
$$

\end{itemize}
\end{defi}
Moreover, the following notion of strict comparison subsolution (supersolution) plays a fundamental role, at least in the Euclidean setting, see \cite{D} and \cite{DFS_apdes}. Here below we state it in the framework of Carnot groups.
\begin{defi} We say that a function $v\in C(\Omega)$ is a strict comparison subsolution (supersolution) to \eqref{two_phase_Carnot_x} if: $v\in C^{2}(\overline{\Omega^+(v)})\cap C^{2}(\overline{\Omega^-(v)})$ and
\begin{itemize}
\item[(i)] $\Delta_{\mathbb{G}} v>f$ (resp. $\Delta_{\mathbb{G}} v<f$) in a viscosity sense in $\Omega^+(v)\cup\Omega^-(v);$
\item[(ii)] for every  $x_0\in \mathcal{F}(v),$ if $x_0$ is not characteristic for $\mathcal{F}(v),$ then 
$$
(v_\nu^+(x_0))^2-(v_\nu^-(x_0))^2> 1\quad (\mbox{resp.}\quad (v_\nu^+(x_0))^2-(v_\nu^-(x_0))^2<1.
$$
\end{itemize}
\end{defi}
As a consequence, we obtain the following result.
\begin{teo}
None strict viscosity subsolution $v$ of \eqref{two_phase_Carnot_x} can touch a solution $u$ from below at none point in $\mathcal{F}(u)\cap\mathcal{F}(v)$ that is noncharacteristic for $\mathcal{F}(v)$. Analogously, none strict comparison supersolution $v$  of \eqref{two_phase_Carnot_x} can  touch a viscosity solution $u$ from above at points belonging to $\mathcal{F}(u)\cap\mathcal{F}(v)$ that are noncharacteristic for $\mathcal{F}(v).$
\end{teo}
\begin{proof}
It follows by the definitions of solution and strict sub/super-solution in $\mathbb{G}$.
\end{proof}
\begin{cor}
Let $v$ and $u$ be respectively a strict subsolution and a solution of \eqref{two_phase_Carnot_x} in $\mathbb{G}.$ If $v\leq u$ in $\Omega$ and $\mathcal{F}(v)$ is a noncharacteristic set then $v<u$ in $\Omega.$

Let $w$ and $u$  be respectively a strict supersolution and a solution of\eqref{two_phase_Carnot_x} in $\mathbb{G}.$ If $w\geq u$ in $\Omega$ and $\mathcal{F}(w)$ is a noncharacteristic set, then $w>u$ in $\Omega.$
\end{cor}
\begin{proof}
Suppose that strict subsolution of \eqref{two_phase_Euc} such that  $v\leq u.$ Then such point $x_0$ can not be inside $\Omega^+(u)\cup\Omega^{-}(u)$ because, on the contrary, from
$$
\Delta_{\mathbb{G}} v-\Delta_{\mathbb{G}} u\geq f(x)-f(x)=0
$$
in $\Omega^+(u)\cup\Omega^{-}(u)$
and $v-u$ realizing a maximum at $x_0$ we would introduce a contradiction with the maximum principle. Then this contact point  $x_0\in\mathcal{F}(u)\cap\mathcal{F}(v),$ and, by the definition of strict subsolution, this fact can not happen. 
\end{proof}

As a consequence it might exist solutions $u,v$ of \eqref{two_phase_Carnot_x} such that $v\leq u,$ $u\not \equiv v$ but $u,$ $v$ might touch in a characteristic point $x_0\in\mathcal{F}(u)\cap\mathcal{F}(v).$ In fact it is well known that a Hopf maximum principle in the Heisenberg group formulated simply substituting to the normal derivative at a boundary point the intrinsic (horizontal) normal derivative  fails, since they may exist characteristic points on a $C^1$ surface. For instance, sets with genus $0$ (without holes) having smooth boundary have always characteristic points belonging to the boundary. As a consequence, they can not exist solutions of \eqref{two_phase_Euc} satisfying flux condition  pointwise on the free boundary, when $\mathcal{F}(u)$ is the boundary of a set of genus $0.$ 

Here we give some examples of solutions in $\mathbb{H}^1. $ Let $u$ be a solution of a two phase problem \eqref{two_phase_Euc} in a set $A\subset \mathbb{R}^2$  satisfying the same condition $|\nabla u^+|^2-|\nabla u^-|^2=1$ (in the Euclidean setting) on $\mathcal{F}(u):=A\cap\partial A(u)$. Then $\tilde{u}(x,u,t)=u(x,y)$ is a solution of  \eqref{two_phase_Carnot_x} in the cylinder $\Omega=A\times (a,b),$ when $\mathbb{G}=\mathbb{H}^1.$

In the case of the $p(x)-$Laplace operator characteristic points do not exist. So that the definition of solution of the simpler problem \eqref{two_phase_p(x)}, in the viscosity sense, is the following one, keeping in mind that
we denote by $n$ the normal to $\mathcal{F}(v)$ at $x_0\in \mathcal{F}(v)$ and, by $v_n^+(x_0), \quad v_n^-(x_0)$  we denote the normal derivatives with respect to the inner normal $n$ to $\Omega^+(v)$ and to $\Omega^{-}(v)$ respectively. 
\begin{defi}\label{defi_plap_x}
Let $u\in C(\Omega).$ We say that $u$ is a solution to \eqref{two_phase_p(x)} if:
\begin{itemize}
\item[(i)] $\Delta_{p(x)} u=f$ in a viscosity sense in $\Omega^+(u)$ and $\Omega^-(u);$
\item[(ii)] for every $x_0\in \mathcal{F}(u)$ and for every function $v\in C(B_\varepsilon(x_0))$, $\varepsilon >0$ such that $v\in C^{2}(\overline{B^+(v)})\cap C^{2}(\overline{B^-(v)}),$ being $B:= B_\varepsilon(x_0)$  and $\mathcal{F}(v)\in C^2$ and $\nabla v(x_0)\not=0,$ 

if $v$ touches $u$ from below (resp. above) at $x_0\in \mathcal{F}(v),$ then 
$$
(v_n^+(x_0))^2-(v_n^-(x_0))^2\leq 1\quad (\mbox{resp.}\quad (v_n^+(x_0))^2-(v_n^-(x_0))^2\geq 1).
$$
\end{itemize}
\end{defi}
In this case, even if we consider only "non-degenerate" points where $\nabla v\not=0$ on $\mathcal{F}(u),$ the Hopf maximum principle holds in the classical sense, so that, introducing the following strict comparison notion of subsolution/supersolution,
\begin{defi} \label{defi_plap_x_strict} $v\in C(\Omega)$ is a strict comparison subsolution (supersolution) to \eqref{two_phase_p(x)} if: $v\in C^{2}(\overline{\Omega^+(v)})\cap C^{2}(\overline{\Omega^-(v)})$ and
\begin{itemize}
\item[(i)] $\Delta_{p(x)} v>f$ (resp. $\Delta_{p(x)} v<f$) in a viscosity sense in $\Omega^+(v)\cup\Omega^-(v);$
\item[(ii)] for every  $x_0\in \mathcal{F}(v),$ if $\nabla v(x_0)\not=0$, then 
$$
(v_n^+(x_0))^2-(v_n^-(x_0))^2> 1\quad (\mbox{resp.}\quad (v_n^+(x_0))^2-(v_n^-(x_0))^2<1.
$$
\end{itemize}
\end{defi}
As a consequence we obtain the following result.
\begin{teo}
None strict viscosity subsolution $v$ of \eqref{two_phase_p(x)} can touch a solution $u$ from below. Analogously, none strict comparison supersolution $v$  of \eqref{two_phase_p(x)} can  touch a viscosity solution $u$ from above.
\end{teo}
\begin{proof} The proof immediately follows applying the definitions \eqref{defi_plap_x}, \eqref{defi_plap_x}, because of  inner maximum principle and via the Hopf maximum principle since, in the last case, the gradient on that contact boundary points can not be $0.$ 
\end{proof}

\end{document}